\documentclass{amsart}

\usepackage{amsfonts,amssymb,verbatim,amsmath,amsthm,latexsym,textcomp,amscd}
\usepackage{latexsym,amsfonts,amssymb,epsfig,verbatim}
\usepackage{amsmath,amsthm,amssymb,latexsym,graphics,textcomp}
\usepackage{graphicx}

\usepackage{url}
\usepackage{psfrag}
\usepackage{amsmath,amsthm,amsfonts,amssymb,mathrsfs,bm,graphicx,stmaryrd}
\usepackage{mathtools}
\usepackage[usenames]{color}
\usepackage{hyperref}
\usepackage [latin1]{inputenc}

\usepackage{tikz}
\usetikzlibrary{shapes,arrows}

\begin{document}

\tikzstyle{decision} = [diamond, draw, fill=gray!20, 
    text width=4.5em, text badly centered, node distance=3cm, inner sep=0pt]
\tikzstyle{block} = [rectangle, draw, fill=gray!20, 
    text width=10em, text centered, rounded corners, minimum height=3em]
\tikzstyle{line} = [draw, -latex']
\tikzstyle{cloud} = [draw, ellipse,fill=gray!20, node distance=3cm,
    minimum height=2em]

\newtheorem{theorem}{Theorem}[section]
\newtheorem{prop}[theorem]{Proposition}
\newtheorem{assume}[theorem]{Assumption}
\newtheorem{lemma}[theorem]{Lemma}
\newtheorem{cor}[theorem]{Corollary}
\newtheorem{definition}[theorem]{Definition}
\newtheorem{conj}[theorem]{Conjecture}
\newtheorem{claim}[theorem]{Claim}
\newtheorem{qn}[theorem]{Question}
\newtheorem{defn}[theorem]{Definition}
\newtheorem{defth}[theorem]{Definition-Theorem}
\newtheorem{obs}[theorem]{Observation}
\newtheorem{rmk}[theorem]{Remark}
\newtheorem{ans}[theorem]{Answers}
\newtheorem{slogan}[theorem]{Slogan}
\newtheorem{corollary}[theorem]{Corollary}
\newtheorem{proposition}[theorem]{Proposition}
\newtheorem{observation}[theorem]{Observation}
\newtheorem{property}{Property}[subsection]
\newtheorem{remark}[theorem]{Remark}
\newtheorem{example}[theorem]{Example}

\newtheorem{question}{Question}
\newtheorem{maintheorem}{Theorem}
\newtheorem{maincoro}[maintheorem]{Corollary}
\newtheorem{mainprop}[maintheorem]{Proposition}

\newcommand{\bluecomment}[1]{\textcolor{blue}{#1}}
\newcommand{\boundary}{\partial}
\newcommand{\hhat}{\widehat}
\newcommand{\C}{{\mathbb C}}
\newcommand{\B}{{\mathbb B}}
\newcommand{\Ga}{{\Gamma}}
\newcommand{\G}{{\Gamma}}
\newcommand{\s}{{\Sigma}}
\newcommand{\PSL}{{PSL_2 (\mathbb{C})}}
\newcommand{\pslc}{{PSL_2 (\mathbb{C})}}
\newcommand{\pslr}{{PSL_2 (\mathbb{R})}}
\newcommand{\Gr}{{\mathcal G}}
\newcommand{\integers}{{\mathbb Z}}
\newcommand{\natls}{{\mathbb N}}
\newcommand{\ratls}{{\mathbb Q}}
\newcommand{\reals}{{\mathbb R}}
\newcommand{\proj}{{\mathbb P}}
\newcommand{\lhp}{{\mathbb L}}
\newcommand{\tube}{{\mathbb T}}
\newcommand{\cusp}{{\mathbb P}}
\newcommand\AAA{{\mathcal A}}
\newcommand\HHH{{\mathbb H}}
\newcommand\BB{{\mathcal B}}
\newcommand\CC{{\mathcal C}}
\newcommand\DD{{\mathcal D}}
\newcommand\EE{{\mathcal E}}
\newcommand\FF{{\mathcal F}}
\newcommand\GG{{\mathcal G}}
\newcommand\HH{{\mathcal H}}
\newcommand\II{{\mathcal I}}
\newcommand\JJ{{\mathcal J}}
\newcommand\KK{{\mathcal K}}
\newcommand\LL{{\mathcal L}}
\newcommand\MM{{\mathcal M}}
\newcommand\NN{{\mathcal N}}
\newcommand\OO{{\mathcal O}}
\newcommand\PP{{\mathcal P}}
\newcommand\QQ{{\mathcal Q}}
\newcommand\RR{{\mathcal R}}
\newcommand\SSS{{\mathcal S}}
\newcommand\TT{{\mathcal T}}
\newcommand\UU{{\mathcal U}}
\newcommand\VV{{\mathcal V}}
\newcommand\WW{{\mathcal W}}
\newcommand\XX{{\mathcal X}}
\newcommand\YY{{\mathcal Y}}
\newcommand\ZZ{{\mathcal Z}}
\newcommand{\iid}{{i.i.d.\ }}
	\renewcommand{\ae}{{a.e.\ }}
\newcommand\CH{{\CC\Hyp}}
\newcommand{\Chat}{{\hat {\mathbb C}}}
\newcommand\MF{{\MM\FF}}
\newcommand\PMF{{\PP\kern-2pt\MM\FF}}
\newcommand\ML{{\MM\LL}}
\newcommand\PML{{\PP\kern-2pt\MM\LL}}
\newcommand\GL{{\GG\LL}}
\newcommand\Pol{{\mathcal P}}
\newcommand\half{{\textstyle{\frac12}}}
\newcommand\Half{{\frac12}}
\newcommand\Mod{\operatorname{Mod}}
\newcommand\Area{\operatorname{Area}}
\newcommand\ep{\epsilon}
\newcommand\Hypat{\widehat}
\newcommand\Proj{{\mathbf P}}
\newcommand\U{{\mathbf U}}
 \newcommand\Hyp{{\mathbf H}}
\newcommand\D{{\mathbf D}}
\newcommand\Z{{\mathbb Z}}
\newcommand\R{{\mathbb R}}
\newcommand\Q{{\mathbb Q}}
\newcommand\E{{\mathbb E}}
\newcommand\EXH{{ \EE (X, \HH_X )}}
\newcommand\EYH{{ \EE (Y, \HH_Y )}}
\newcommand\GXH{{ \GG (X, \HH_X )}}
\newcommand\GYH{{ \GG (Y, \HH_Y )}}
\newcommand\ATF{{ \AAA \TT \FF }}
\newcommand\PEX{{\PP\EE  (X, \HH , \GG , \LL )}}
\newcommand{\lct}{\Lambda_{CT}}
\newcommand{\lel}{\Lambda_{EL}}
\newcommand{\lgel}{\Lambda_{GEL}}
\newcommand{\lre}{\Lambda_{\mathbb{R}}}

\newcommand\til{\widetilde}
\newcommand\length{\operatorname{length}}
\newcommand\tr{\operatorname{tr}}
\newcommand\cone{\operatorname{cone}}
\newcommand\gesim{\succ}
\newcommand\lesim{\prec}
\newcommand\simle{\lesim}
\newcommand\simge{\gesim}
\newcommand{\simmult}{\asymp}
\newcommand{\simadd}{\mathrel{\overset{\text{\tiny $+$}}{\sim}}}
\newcommand{\ssm}{\setminus}
\newcommand{\diam}{\operatorname{diam}}
\newcommand{\pair}[1]{\langle #1\rangle}
\newcommand{\T}{{\mathbf T}}
\newcommand{\I}{{\mathbf I}}
\newcommand{\pG}{{\partial G}}
 \newcommand{\oxi}{{[1,\xi)}}
\newcommand{\cg}{\mathcal{G}}

\newcommand{\tw}{\operatorname{tw}}
\newcommand{\base}{\operatorname{base}}
\newcommand{\trans}{\operatorname{trans}}
\newcommand{\rest}{|_}
\newcommand{\bbar}{\overline}
\newcommand{\lbar}{\underline}
\newcommand{\UML}{\operatorname{\UU\MM\LL}}
\newcommand{\EL}{\mathcal{EL}}
\newcommand{\ncox}{{N_C([o,\xi))}}
\newcommand{\qle}{\lesssim}

\def\ind{{\mathbf 1}}
\def\N{\mathbb{N}}
\def\L{\mathbb{L}}
\def\P{\mathbb{P}}
\def\Z{\mathbb{Z}}
\def\R{\mathbb{R}}
\def\S{\mathcal{S}}
\def\X{\mathcal{X}}
\def\E{\mathbb{E}}
\def\l{\ell}
\def\Ups{\Upsilon}
\def\om{\omega}
\def\Om{\Omega}

\newcommand\Gomega{\Omega_\Gamma}
\newcommand\nomega{\omega_\nu}
\newcommand\omegap{{(\Omega,\P)}}
\newcommand\omegapp{{(\Omega',\P)}}

\DeclarePairedDelimiter\floor{\lfloor}{\rfloor}
\newcommand{\cF}{\mathcal{F}}
\newcommand{\cD}{\mathcal{D}}
\newcommand{\cZ}{\mathcal{Z}}
\newcommand{\cf}{\mathcal{F}}
\newcommand{\cA}{\mathcal{A}}
\newcommand{\cB}{\mathcal{B}}
\newcommand{\cC}{\mathcal{C}}
\newcommand{\cT}{\mathcal{T}}
\newcommand{\rb}{\mathfrak{A}}
\newcommand{\Var}{\hbox{Var}}
\newcommand{\ee}{\hbox{e}_1}
\newcommand{\cd}{\mathcal{D}} 
\newcommand{\ce}{\mathcal{E}}
\setcounter{tocdepth}{1}

\title{Random Trees in Hyperbolic FPP}

\author{Riddhipratim Basu}
\address{Riddhipratim Basu, International Centre for Theoretical Sciences, Tata Institute of Fundamental Research, Bengaluru, India}

\email{rbasu@icts.res.in}

\author{Mahan Mj}
\address{Mahan Mj, School
of Mathematics, Tata Institute of Fundamental Research. 1, Homi Bhabha Road, Mumbai-400005, India}

\email{mahan@math.tifr.res.in}

\thanks{RB is partially supported by a SERB MATRICS grant  from DST Govt.\ of India (MTR/2021/000093), and DAE project no.\ RTI4019 via ICTS.  MM is partly supported by  a DST JC Bose Fellowship,  and by  the Department of Atomic Energy, Government of India, under project no.12-R\&D-TFR-5.01-0500.  Both authors were supported in part by an endowment of the Infosys Foundation. This project began at the International Centre for Theoretical Sciences (ICTS) during  the program - Probabilistic Methods in Negative Curvature (Code: ICTS/PMNC2019/03). We  gratefully acknowledge the support of ICTS}
\subjclass[2010]{60K35, 82B43,  20F67 (20F65, 51F99, 60J50) } 

\keywords{first passage percolation, hyperbolic group,  coalescence}

\date{\today}

 \begin{abstract}
We start by surveying known properties of first passage percolation (FPP) geodesics on a Gromov-hyperbolic group $G$. Due to a coalescence phenomenon established in earlier work of the authors, a random tree $T(\xi,\omega)$ consisting of infinite random geodesics to a point $\xi$ on the Gromov boundary $\pG$ emerges naturally. These trees have a rich random geometry that we explore further in this paper. 
\end{abstract}

\maketitle

\tableofcontents

\section{Introduction}\label{sec-intro} First passage percolation (FPP) is a well-studied model of a random metric on a graph, obtained by assigning independent and identically distributed random lengths to each edge. Over the last 60 years, FPP has been extensively studied, primarily on Euclidean background geometry. Of late, there has been interest in studying FPP on different background geometries, in particular on Gromov-hyperbolic graphs \cite{benjamini-tessera,bz-tightness,BM19,BM24}. This partly expository paper  starts by surveying known theorems for FPP on Cayley graphs of Gromov hyperbolic groups, primarily the results established in \cite{BM19, BM24}. An important fact established in  \cite{BM19} is coalescence of random geodesics directed towards a common point  $\xi$ in the Gromov boundary  $\pG$ of $G$: given a hyperbolic group $G$ and
 for every $\xi$ in the Gromov boundary  $\pG$ of $G$, any two geodesic rays (in the random first passage percolation metric) directed towards $\xi$  almost surely coalesce. As observed in \cite{BM24}, for almost every realization $\omega$ of the random environment, the collection of FPP geodesics directed towards 
$\xi$ forms a random tree $T(\xi,\omega)$,  referred to in that paper and below as the \emph{backward random tree} in the direction $\xi$. 

Such geodesic trees are commonly studied for first and last passage percolation models on $\Z^2$. We refer the interested reader to the excellent monograph \cite{ADH} for details on what was known classically, and to \cite{H05,hoffman, dh, DH17,AH16} for recent progress in  planar  Busemann functions. 
It was established by Damron and Hansen in \cite[Theorem 1.10(4)]{dh} that, under mild conditions,  ``backward clusters" are finite. For  finer properties in the case of the exactly solvable exponential last passage percolation, see \cite{BBB23}. 
One of the motivations for studying this in the context of hyperbolic FPP is to study the effect of the ambient geometry on such random trees. We showed in \cite{BM24}, that under mild conditions on FPP on hyperbolic groups and for a fixed $\xi\in \pG$, the backward random tree $T(\xi,\omega)$ is almost surely \emph{infinite in all directions}, i.e., for every $\xi'\ne \xi$, $T(\xi,\omega)$ contains a bi-infinite geodesic $(\xi',\xi)_{\omega}$. After reviewing the relevant results from \cite{BM19,BM24}, the rest of this paper is devoted to further study of the random trees $T(\xi,\omega)$.

%
%
%
%


\subsection*{Organization of the paper}
The rest of the paper is organized as follows. We start off in Section \ref{sec-prelhypperc} by recalling the basic set up of first passage percolation on hyperbolic groups and by describing the results about the first order and fluctuation behavior of the FPP metric from \cite{BM19}. In Section \ref{sec-btree}, we recall the relevant results about direction and coalescence of infinite geodesics in hyperbolic FPP established in \cite{BM19,BM24}. The remaining sections contain some new results about the trees and some related statistics. In Section \ref{sec-coalescedistrbn} we show that the point of coalescence between two random geodesics in the same boundary direction is typically close to the deterministic centroid of the three endpoints. {As a consequence, we show (see Corollary~\ref{cor-decorr-tree}) that the backward sub-trees of  $T(\xi,\omega)$ rooted at far enough points along a horosphere are contained in disjoint cones with large probability}. In  Section \ref{sec-lambda}, we study the growth of the backward trees, and show that the backward subtrees are finite with positive probability, and study a few related questions. 
{In particular, we 
reprove a slightly weaker version of \cite[Theorem A]{bjq25} in Corollary~\ref{cor-nonhyp}.}

\section{Preliminaries on  hyperbolic first passage percolation}\label{sec-prelhypperc} 
Let $(\Ga,d)$ be a Gromov-hyperbolic graph equipped with the graph distance metric $d$. Let $V=V(\Gamma)$ and $E=E(\Gamma)$ denote its vertex and edge sets respectively. Consider the $E$-fold product of the Borel $\sigma$-algebra on $\Omega=[0,\infty)^{E}$, denoted $\mathscr{B}$. 
 Let $\Omega= [0, \infty)^E$ equipped with the Borel $\sigma$-algebra $\mathscr{B}$. Let $\rho$ be a probability measure  on $[0, \infty)$. Let  $\P=\rho^{\otimes E}$ be the product measure. Let $\mathscr{B}^*$ denote the $\P$-completion of $\mathscr{B}$. We shall be working with the probability space $(\Omega, \mathscr{B}^*, \P)$.  We shall denote elements of $\Omega$  by $\omega=\{\omega(e)\}_{e\in E}$.  The \emph{edge length} random variables $X_{e}: \Omega \to [0,\infty)$ given by $X_{e}(\omega)=\omega(e)$ for $e\in E$ will be \iid with law $\rho$ under $\P$. 

\begin{defn}\label{def-fpp}
	Let $\gamma = \{e_1, \cdots e_k\}$ be an edge path. For $\omega \in \Omega$, the $\omega-$length of $\gamma$ is defined to be $$\ell_{\omega}(\gamma):= \sum_{e\in \gamma } \omega(e).$$ Define $$d_{\omega} (x,y) := \inf_\gamma \ell_{\omega}(\gamma),$$ where $\gamma$ ranges over edge paths connecting $x, y \in V$. The random variable $T(x,y)$ defined on $\Omega$ by $$T(x,y)(\omega)=d_\omega (x,y)$$ will be called the {\bf first passage time} between $x$ and $y$.
\end{defn}

We shall assume throughout that $\rho$ is continuous, i.e., it does not have any atoms; in particular it does not put any mass on $0$. Under such hypotheses it is easy to see that paths attaining the first passage time exist and are unique almost surely. 

\begin{defn}
	A path that realizes $d_{\omega}(x,y)$ will be called an {\bf $\omega-$ geodesic}, denoted $[x,y]_{\omega}$. Observe that under our hypothesis on $\rho$, for $\P$-a.e. $\omega\in \Omega$, there is a unique $\omega$-geodesic between each pair of points in $\Ga$. For fixed vertices $x,y\in \Ga$, this ($\P$-a.e. well-defined) random path $\Upsilon(x,y)$ (i.e., $\Upsilon(x,y)(\omega)$ denotes the $\omega$-geodesic between $x$ and $y$) will be called the {\bf FPP-geodesic} between $x$ and $y$. 	
\end{defn}

We shall further assume throughout the rest of the paper that the edge length distribution $\rho$ has super exponentially decaying tails. 

\begin{assume}\label{assume-subexp}
	Suppose that the edge weights $X_e$ are i.i.d.\ non-negative random variables with distribution $\rho$ such that for {all} $a > 0$, $\int_0^\infty e^{ax} d\rho(x) $ is finite. 
\end{assume}

\subsection{Velocity}\label{sec-vel}
For the purposes of this paper, fix a hyperbolic group $G$. Let $\Gamma$ denote a Cayley graph of $G$ with respect to a finite symmetric generating set. Let $1\in \Gamma$ denote the identity element of $G$.
Let $\partial G$ denote the Gromov boundary of $G$. We equip $\pG$ with the Patterson-Sullivan measure $\nu$. This measure $\nu$ may be thought of as a weak limit 
of the uniform measures on balls. In this section, we summarize two of the main theorems of \cite{BM19}:
\begin{enumerate}
\item In almost every boundary direction, first passage time grows linearly with a fixed velocity.
\item Variance of the first passage time grows linearly in the word distance along a word geodesic. 
\end{enumerate}

\begin{defn}\label{def-vel}
Let $\xi \in \pG$. The velocity  in the direction of $\xi$ is defined as
${v} (\xi) := \lim_{n \to \infty} \frac{\E(\T(1, x_n))}{n},$
where $x_n\to \xi$ with $d(1,x_n)=n$ is any sequence \emph{provided such a limit exists.}
\end{defn}

The following theorem says that velocity exists and is constant $\nu-$a.e.

\begin{theorem}\label{thm-vexists}\cite[Theorem 5.1]{BM19}
	For $\nu$-a.e.\ $\xi \in \pG$,  the velocity $v(\xi)$ in the direction of $\xi$ exists. Further,  $v(\xi)$ is constant $\nu$-almost everywhere in $\pG$.
\end{theorem}

A word about the proof is in order as Theorem~\ref{thm-vexists} is the only place in \cite{BM19} where the group structure of $G$ is used; for the other theorems in \cite{BM19}, $G$ can be replaced by a proper Gromov-hyperbolic graph.

The fact that $G$ is a hyperbolic group allows us to use a theorem of Cannon \cite{cannon-conetype,cannon-book}  that furnishes  an automatic structure. We refer the reader to \cite[Section 3]{BM19} for details. Fix a finite symmetric generating set $S$ of $G$. Then $S$ generates $G$ as a semigroup. The main theorem of \cite{cannon-conetype} furnishes a finite state automaton $\GG$ (i.e.\   a finite digraph with directed edges labeled by $S$, and a distinguished initial state) such that the language accepted by $\GG$ is a regular geodesic language. Moreover, for any $g \in G$, there is a unique geodesic word $w_g$ accepted by $\GG$. Let $M$ denote the transition matrix of $\GG$ and $\lambda > 0$ denote its Perron-Frobenius eigenvalue. There are finitely many components of the directed graph $\GG$ whose Perron-Frobenius eigenvalue equals $\lambda$. These are called the \emph{maximal components} of $\GG$.
It turns out that each directed edge path in $\GG$ eventually lands in at most one maximal component, see \cite[Proposition 4.10]{calegarifujiwara}, \cite[Proposition 3.12]{calegarimaher} or \cite[Proposition 3.12]{BM19}. 
We formalize the above discussion as follows. Let $\PP^0$ denote the collection of directed paths in $\GG$ starting at the initial state 1. These are identified with geodesics in the Cayley graph $\Gamma$ of $G$ with respect to the generating set $S$.
Let $\PP^0_n$ denote the subset of $\PP^0$ consisting of directed paths of length $n$. Let $e:\PP^0 \to G $ denote the evaluation map  sending a labeled path in $\PP^0$ to the corresponding element in $G$.
Further, let  $\PP_\infty(C)$
denote the collection of infinite paths contained in the maximal component $C$. Let $\nu$ denote the Patterson-Sullivan measure on $\pG$ \cite{coornert-pjm}. {We shall  use $\til \nu$ to denote the lift of the Patterson-Sullivan measure to the path space $\PP^0_{\infty}$}. Then (see \cite[Corollary 3.16]{BM19}) we have the following consequence  summarizing the above discussion.
\begin{cor}\label{cor-eventuallyunique}
	Let $\SSS$ denote the shift on the path space $\PP^0$.
	For $\til\nu-$a.e.\ $\sigma \in \PP^0_\infty$, there exists $M \in \natls$ and  a  maximal component $C$ such that for all $i\geq M$, $\SSS^i (\sigma) \in  \PP_\infty(C)$.
\end{cor}

Let $\til \GG$ denote the locally finite  rooted tree encoding $\PP^0$, the set of all  directed paths in $\GG$. Let $\partial \til \GG$ denote its Gromov boundary, so that the set of points $\eta \in \partial \til \GG$ is in natural bijective correspondence with $\PP^0$.
The following Lemma gives us a way of going between $\pG$ and $\PP^0$.

\begin{lemma}\label{lem-ctfsa}\cite[Lemma 3.5.1]{calegari-notes}\cite[Lemma 3.15]{BM19}
	The  map $e: \PP^0_n \to G$ admits a continuous extension  $\bbar e: \PP^0_\infty \to \partial G$. Further, there exists $M \in \natls$ such that for all 	$\xi \in \partial G$, $|\bbar{ e}^{-1} (\xi)| \leq M$.
\end{lemma}

The key group-theoretic conclusion that is drawn in \cite{BM19} is the existence of frequencies. Let $\PP_\infty^{\max} \subset \PP^0_\infty$ denote the collection of  semi-infinite geodesic words satisfying the conclusion of Corollary~\ref{cor-eventuallyunique}, i.e.\ directed paths eventually landing in a maximal component. Fix $\sigma \in \PP_\infty^{\max}$. Then for a finite geodesic word $w$, the \emph{frequency} $f_w (\sigma)$ is a normalized measure giving the proportion of occurrence of $w$ in $\sigma$. Then \cite[Lemma 3.20]{BM19} states that for $\til\nu-$almost all $\sigma \in \PP_\infty^{\max}$, the 
frequency $f_w (\sigma)$ exists. Let $\PP_\infty^{\prime}$ denote this full measure subset of $\PP_\infty^{\max}$ (and hence a full measure subset of $\PP^0_\infty$) for which the 
frequency $f_w (\sigma)$ exists.

Once the existence of frequencies is established, the argument in \cite{BM19} is a combination of  probabilistic and geometric  ideas (in \cite[Sections 4, 5]{BM19}) and does not involve the group structure of $G$ any further. There are two geometric facts used here:
\begin{itemize}
\item the Morse Lemma, stating that quasigeodesics track geodesics, and
\item the existence of quasiconvex hyperplanes in $\Gamma$.
\end{itemize}
Both hold in general hyperbolic metric spaces, as do the relevant probabilistic estimates.

The upshot of this is the following statement that establishes the existence part of Theorem~\ref{thm-vexists}.

\begin{prop}\label{prop-vexists}
 Let $\PP_\infty^{\prime}$ denote, as above, the $\til\nu-$full measure subset of $\PP^0_\infty$ for which the 
frequency $f_w (\sigma)$ exists for any finite geodesic word $w$.
Let $\bbar{e}(\PP_\infty^{\prime}) \subset \pG$ denote the set of points in $\pG$ that are ideal end-points of $\sigma \in \PP_\infty^{\prime}$ (note that this is a $\nu-$full measure subset of $\pG$ equipped with the Patterson-Sullivan measure
$\nu$).
Then, for $\xi \in \bbar{e}(\PP_\infty^{\prime})$, the velocity $v(\xi)$ in the direction of $\xi$ exists.
\end{prop}
Note that Proposition~\ref{prop-vexists} identifies a precise full measure 
subset of $\pG$ equipped with the Patterson-Sullivan measure $\nu$ for which velocity exists, i.e.\ the existence part of Theorem~\ref{thm-vexists} holds. A key fact that feeds into 
the statement of Proposition~\ref{prop-vexists} is the following: Let $\sigma, \sigma' \in \PP_\infty^{\prime}$ be such that  $\bbar{e} (\sigma) = \bbar{e} (\sigma') = \xi$. Then
$v(\xi)$ is independent of the choice of the path $\sigma$ or $ \sigma'$ used. This is because ${e} (\sigma)$ and ${e} (\sigma')$ are semi-infinite geodesic rays in the Cayley graph $\Gamma$ of $G$ both converging to $\xi$. Hence ${e} (\sigma)$ and ${e} (\sigma')$ lie within a $2 \delta$ neighborhood of each other (where $\delta$ is the hyperbolicity constant of $\Gamma$). It now follows from  \cite[Remark 5.3]{BM19}  that $v(\xi)$ is independent of the choice of  $\sigma, \sigma'$.

 To conclude that $v(\xi)$ is constant $\nu-$almost everywhere, we use the fact that $G$ acts ergodically on $(\pG, \nu)$. However, to apply ergodicity of the $G$ action, we need to observe first that  $v(\xi)$ is a measurable function on $\pG$. Recall that $\bbar{e}( \PP_\infty^{\prime})$ is a full measure subset and it suffices that $v(\xi)$ is well-defined on this set.

To establish measurability of $v$, we start off in the path space $\PP_\infty^{0}$.  In \cite[Proposition 3.16]{BM19}, the existence of frequencies is established. This is really a standard fact about finite state Markov chains and is proved in detail in \cite[Appendix A]{BM19}.
The proof there establishes that the frequency function not only exists but is measurable. From this, it follows that $v(\sigma)$ is a measurable function on the full measure subset $\PP_\infty^{\prime}$ of $\PP_\infty^{0}$
{that is constant on fibers of $\bar{e}$}. We emphasize that
\cite[Definition 3.10]{BM19} and the discussion there  gives us a well-defined way of lifting the Patterson-Sullivan measure $\nu$ on $\pG$ to 
a measure $\til \nu$ on the path space $\PP^0_\infty$, so that {the push-forward of $\til \nu$ under $\bbar{e}$ equals $\nu$, and $v$ is constant on the fibers of 
$\bbar{e}$ in 
 $\PP_\infty^{\prime}$. }).

It remains to show that $v(\sigma)$ `pushes forward' to a measurable function $\bbar{v}(\xi)$ on $\pG$ under $\bbar{e}$. 
By Lemma~\ref{lem-ctfsa}, $\bbar{e}$ is continuous and finite-to-one, with
the cardinality of $\bbar{e}^{-1}(\xi)$ uniformly bounded. Since 
$\bbar{e}$ is a continuous map between compact sets, it is a quotient map. We spell this out. Let $h, \bbar{h}$ be continuous functions (to $\R$) on $\PP^0_\infty$ and $\pG$ respectively, such that $h =   \bbar{h}\circ\bbar{e}$. Then $h$ is continuous if and only if $\bbar{h}$ is continuous.
The same holds for closed (and hence compact) subsets of $\PP^0_\infty$ and their images in $\pG$ under $\bbar{e}$.

To show that $v(\sigma)$ `pushes forward' to a measurable function $\bbar{v}(\xi)$ on $\pG$ under $\bbar{e}$, we use Lusin's theorem. 
Let $\ep >0$.
{By (the forward direction of) Lusin's  Theorem $v$ is continuous on a closed subset $K_\ep$ of measure at least $(1-\ep)$}.
Since $\bbar{e}$ is a quotient map, {$\bbar{v} \circ \bbar{e}$ is continuous on the image of this  set $K_\ep$. Note that this part of the argument does not require $K_\ep$ to be saturated. However, since $\bbar{v} \circ \bbar{e}$ is continuous on the image of $K_\ep$, and since 
$\bbar{e}$ is a quotient map,
it follows that $v$ is in fact continuous on the saturated set $\bbar{e}^{-1}\circ\bbar{e}(K_\ep)$. Thus, without loss of generality, we may assume that the closed subset $K_\ep$ we started with, is saturated}. The image $\bbar{e}(K_\ep)$ is of measure at least $(1-\ep)$
since $\bbar{e}$ preserves measures. An alternate way to see that the measure of the image of $K_\ep$ is close to 1 is by noting that
 small diameter sets go to small diameter sets 
under $\bbar{e}$ and that the Patterson-Sullivan measure is Ahlfors regular. Since $\bbar{e}$ is a quotient map, it follows that   $\bbar{v}$ is continuous on a set of measure close to 1.
Since $\ep>0$ is arbitrary it follows from  (the converse direction of) Lusin's theorem that $\bbar{v}$  is measurable on $\pG$ as desired.

Finally, {observe that by the translation invariance of $\P$, we have $\E T(1,x_n)=\E T(g,gx_n)$, and it is easy to conclude from there using the triangle inequality that $v(\xi)=v(g\xi)$ provided one of them exists, and hence $v$ is $G$-invariant}. Therefore, as noted before, we can apply  ergodicity of the $G$ action on $(\pG, \nu)$ to conclude that $v$ is constant almost everywhere, concluding the proof of Theorem~\ref{thm-vexists}.

\subsection{Bi-velocity} 
Next we state a simple extension of Theorem \ref{thm-vexists}. The proof is postponed to Section \ref{s:bivel}.


\begin{defn}\label{def-bivel}
	Given $x, y \in \partial G, x\neq y$, let $(x,y)$ be any bi-infinite geodesic from $x$ to $y$. Let $x_n, y_n \in (x,y)$ be such that $x_n \to x; \, y_n \to y$. Then define the bi-velocity $$v(x,y):=
	\lim_{n \to \infty} \frac{\E (T(x_n,y_n))}{d(x_n,y_n)}, $$ if the limit exists and is {independent of the choice of the bi-infinite geodesic}.
\end{defn}

We need to introduce  some notation. 

\begin{itemize}
	\item  $\partial^2 G:= \{(\xi, \xi'): \xi, \xi' \in \pG; \, \xi\neq \xi'\}$.
	\item  {$\partial^3 G:= \{(\xi, \xi',\xi''): \xi, \xi',\xi'' \in \pG; \, \xi\neq \xi'\neq \xi''\ne \xi \}$.}
\end{itemize}

\begin{prop}\label{prop-bivel} Let $v$ denote the constant velocity 
	in Theorem~\ref{thm-vexists}.
	For a.e.\ $ (x,y) \in \partial^2 G$  (with respect to the $\nu \times \nu$ on $\partial^2 G$), 	$v(x,y)=v$.
\end{prop}

The proof of this proposition requires a hyperplane construction from \cite{BM19} and will be given later in Section \ref{sec-coalescedistrbn} once this construction is recalled.

\subsection{Linear growth of variance}
Theorem \ref{thm-vexists} establishes the first order behavior of the first passage times. The next natural question is to investigate the fluctuations of the first passage times. To this end, the following theorem was established in \cite{BM19}.

\begin{theorem}[{\cite[Theorem 8.1]{BM19}}]
	\label{t:linvar}
	For $\xi\in \pG$ and a geodesic ray $[1,\xi)=\{1,x_1,x_2,\ldots,x_n, \ldots\}$, there exists $0<C_1<C_2<\infty$ such that 
	$$ C_1 n\leq \mathrm{Var} (T(1,x_n))\leq C_2 n.$$
\end{theorem}

The above theorem resolved  a conjecture of Benjamini, Tessera and Zeitouni \cite{bz-tightness,benjamini-tessera} affirmatively.

\section{Geodesic rays, coalescence, and geodesic trees}\label{sec-btree}
The aim of this section is to summarize the main results of \cite{BM19,BM24} regarding semi-infinite and bi-infinite random geodesics in hyperbolic FPP and associated geodesic trees. We start with recalling the basic results from \cite{BM19} about the direction of random geodesics.


%
\begin{defn}\label{def-fppray} 
	For  $\om \in \omegap$, a semi-infinite (resp.\ bi-infinite) path $\sigma_\omega$ is said to be an $\om-$geodesic ray (resp.\ a bi-infinite
		$\om-$geodesic) if every finite subpath of   $\sigma_\omega$ is an $\om-$geodesic.
	
	A path $\sigma$ is said to  accumulate on $\xi \in \pG$ if there exist $v_n \in  \sigma$ such that $v_n \to \xi$ as $n \to \infty$. 
		An $\om-$geodesic ray $\sigma_\omega$ accumulating on $\xi\in \pG$  has direction $\xi$ if its only accumulation point  in $\pG$ is $\xi$.
\end{defn}

The following theorem says that almost surely all $\omega$-geodesic rays have a direction and that
every $\xi\in \pG$ is the direction of an $\omega$-geodesic ray.

\begin{theorem}\cite[Theorems 6.6, 6.7]{BM19}
	\label{bmdirexists}
	For $o\in \Gamma$, for a.e. $\omega \in \omegap$, all $\om$-geodesic rays starting at $o$ have a direction $\xi\in \pG$.

	Fix $\xi \in \partial G$ and $x_{n}\in \Ga$ such that $x_{n}\to \xi$. For a.e.\ $\om \in \omegap$ the sequence of $\om-$geodesics $[o,x_n]_\om$ from $o$ to $x_n$ converges (up to extracting a subsequence) to an $\om-$geodesic ray $[o,\xi)_\om$ having direction $\xi$.
\end{theorem}

 \subsection{Coalescence and backward trees}
		Two semi-infinite paths $\sigma_1, \sigma_2 \subset \Gamma$ are said to {\bf coalesce} if beyond some $x_0$, they coincide.
	For FPP on $\Ga$, we showed in \cite{BM19} that semi-infinite geodesic rays in a fixed direction $\xi\in \pG$ almost surely coalesce, i.e.\ the set of edges in the symmetric difference of any two such geodesic rays outside a sufficiently large ball centered at the identity element is empty. More precisely we have the following theorem.

	\begin{theorem}\cite[Lemma 7.8, Theorem 7.9]{BM19}\label{bmthmcoalesce} 
		Given $\xi\in \partial G$, there exists a full measure subset $\Omega_\xi \subset \omegap$ such that for all  $\omega\in \Omega_\xi$, and each $o\in \Ga$,  there exists a unique $\omega$-geodesic ray (denoted $[o,\xi)_{\omega}$) starting from $o$ in direction $\xi$. Further, given any direction $\xi$,	for  all  $\omega\in \Omega_\xi$, and any $o_1, o_2 \in G$
		the  $\om-$geodesic rays $[o_1,\xi)_\om$ and $[o_2,\xi)_\om$ coalesce.
	\end{theorem}
As an immediate consequence of Theorem~\ref{bmthmcoalesce}, we have

\begin{cor}\label{cor-btree} 
		Given $\xi\in \partial G$, there exists a full measure subset $\Omega_\xi \subset \omegap$ such that for all  $\omega\in \Omega_\xi$,
$T(\xi,\omega) = \bigcup_{g \in G} [g,\xi)_\om$ is a tree.
\end{cor}

\begin{defn}\label{def-btree}
We shall refer to $T(\xi,\omega)$ as the \emph{random backward geodesics tree} (backward tree for short) from $\xi$ for
$\om \in \Omega_\xi$. 
\end{defn}

\subsection{Bigeodesics and effective estimates}
 In  \cite{benjamini-tessera}, the existence of bi-geodesics was established for very general edge-weight distributions for hyperbolic FPP. In \cite{BM19,BM24} under the stronger hypothesis Assumption \ref{assume-subexp}, we make some of their estimates effective and prove stronger results about structures of bigeodesics.  

\begin{prop}[{\cite[Proposition 2.1]{BM24}}]
\label{prop-effbt} Fix $1 \in \Gamma$, and $C > 0$. Then 
	for a.e.\ $\om \in \omegap$, there exists
	$R_\omega>0$ such that the following holds:\\ 
	For every sequence $x_n, y_n \to \infty$ such that $[x_n, y_n]$ passes through 
	$N_C (o)$, the $\om-$geodesic $[x_n, y_n]_\omega$ passes
	through the $R_\omega-$neighborhood of $o$.\\
 Further, there exists $c>0, N_0$ such that $$\P[R(\om)\geq N] \leq \exp(-cN)$$ for $N \geq N_0$.
\end{prop}

{Notice that Proposition \ref{prop-effbt} shows that if $o\in [x,y]$, then it is unlikely that $[x,y]_{\omega}$ does not pass through a point near $o$. The complementary result which states that if the $o$ is far away from $[x,y]$ then $[x,y]_{\omega}$ is unlikely to pass close to $o$ is also true. 
\begin{prop}[{\cite[Lemma 2.3]{BM24}}]
    \label{lem-effbt}
    Let $C>0, o \in \Gamma$ be as in Proposition~\ref{prop-effbt}. Suppose Assumption~\ref{assume-subexp} holds. Then, given $c>0$, there exists $R_0=R_0(c,C) >0$ such that for $R \geq R_0$ the following holds.\\
Let $A=A(o,C)$ denote the event that there exist $u, v, w \in \Gamma$ such that
	\begin{enumerate}
	\item $w \in [u,v]_\omega$.
	\item $\Pi_{uv} (w) \in N_{C}(o)$, where $\Pi_{uv}$ denotes nearest-point-projection onto $[u,v]$.
	\item $d(w,o) \geq R$.
	\end{enumerate}
	Then $\P(A) \leq e^{-cR}$. 
\end{prop}
}

Using Proposition \ref{prop-effbt}, we showed that, almost surely there exist random bigeodesics $(\xi,\xi')_{\omega}$ for all $\xi\ne \xi' \in \pG$. As a consequence we get the following information about the ends of the backward tree $T(\xi,\omega)$.

	\begin{defn}\label{def-btreefullend}
	A backward tree $T(\xi,\omega)$  has 
	{\bf complete ends} if for all $\xi'\neq \xi$, there exists $p \in [1,\xi)_\om$ such that $(\xi',p]_\om \subset T(\xi,\omega)$.
\end{defn}
  The next result states that $T(\xi,\omega)$ contains  a union of bi-infinite geodesics $\{(\xi',\xi)_\omega\}$, where $\xi'$ ranges over 
$\partial G \setminus \{\xi\}$. 

\begin{theorem}\label{thm-btree} \cite{BM24}
	For all $\xi \in \partial G$,
	there exists a full measure subset $\Omega_\xi\subset \Omega$  such that  for all $\om \in \Omega_\xi$,  the backward tree
	$T(\xi,\omega)$ has complete ends.
\end{theorem}

\subsection{Random Cannon-Thurston maps and exceptional directions}\label{sec-btstr}

\begin{defn}\label{def-fwdtree}
	Consider FPP on $\Gamma$ with passage times satisfying that $\rho$ has no atoms. For $\P$-a.e. $\omega\in \Omega$, the union of random geodesics $$F(1,\omega):=\bigcup_{g \in G} [1,g]_\omega $$ is a tree, which we refer to as the  \emph{forward random geodesic tree}, or simply the forward tree,  of the FPP on $\Gamma$ with base point $1$. Forward random tree with base point $g\in G$ is defined similarly.  
\end{defn}

We now introduce a key tool from geometric group theory, the Cannon-Thurston map in the random context. This notion was introduced into the study of (Gromov) hyperbolic groups in \cite{mitra-ct,mitra-trees}
based on    the original work of Cannon and Thurston in \cite{CT,CTpub}.
See \cite{mahan-icm} for further background and a survey of results.

\begin{defn}\label{def-ct}
	Let $(X, d_X)$ and  $(Y, d_Y)$  be  hyperbolic metric spaces and let $i: (Y,y_0) \to (X,x_0)$ denote a proper embedding. Let $\partial X, \, \partial Y$ be their (Gromov) boundaries.
	Also, let $\hhat X$ and $\hhat Y$ denote their Gromov compactifications.
	We say that the triple $(X,Y,i)$ admits a Cannon-Thurston map, if
	$i$ extends continuously to 
	$\hhat i: \hhat Y \to \hhat X$. The continuous extension, if it exists, is referred to as a \emph{Cannon-Thurston map} for 
	$i: Y \to X$.
\end{defn}

The following theorem asserts the existence of random Cannon-Thurston maps for both forward and backward trees.
\begin{theorem}\label{thm-randomctexists}\cite[Section 3.1]{BM24} We fix the base-point $1$.
	There exists a full measure subset $\Omega_0 \subset \Omega$ such that
	for all  $\omega \in \Omega_0$, the natural embedding 
	$i_\omega^F: F(1,\omega) \to \Gamma$ admits a Cannon-Thurston map. Further, given $\xi \in \pG$, there exists a full measure set $\Omega_{\xi} \subset \Omega$  such that	for all  $\omega \in \Omega_{\xi}$, the natural embedding $i_\omega^T: T(\xi,\omega) \to \Gamma$ admits a Cannon-Thurston map. 
\end{theorem}

We now define exceptional direction for the forward and backward trees.

 \begin{defn}\label{def-exceptionaldir} Let $G$ be hyperbolic.
	For $\omega \in \Omega$, we say that $z \in \pG$ is an exceptional direction for the  forward random tree  $F(1,\omega)$ if there exist {\emph{distinct}} geodesics
	$[1,z)^1_\omega$ and $[1,z)^2_\omega$ contained in 
	$F(1,\omega)$ such that $z\in \pG$ is the unique accumulation point for both $[1,z)^1_\omega$ and $[1,z)^2_\omega$.
\end{defn}

\begin{defn}\label{def-except}
	For the backward random tree
	$T(\xi,\omega)$, we say that 
	$z\in \pG$ is an \emph{exceptional direction} if $T(\xi,\omega)$ contains {\emph{disjoint}} $\omega-$geodesic rays $[a,z)_\omega$ and 
	$[b,z)_\omega$, both converging to $z \in \pG$. 
\end{defn}

{Note that in Definition~\ref{def-exceptionaldir}, we require forward geodesics to be 
only distinct. Thus, they are allowed to overlap for an initial geodesic segment. On the other hand,  in Definition~\ref{def-except}, backward geodesics are required to have disjoint tails, so that, in particular, the starting points $a, b$
are distinct. }

Theorem~\ref{thm-randomctexists} provides us with a topological tool for investigating exceptional directions in the sense of Definitions \ref{def-exceptionaldir} and ~\ref{def-except}. We provide some brief motivation here.
We know that given $\xi \in \pG$, for almost every  $\om \in \Omega$, 
all $\omega-$geodesics in the direction of $\xi$ coalesce. We refer to this as typical or non-exceptional behavior. Note however, that the full measure subset of $\Omega$ \emph{depends on $\xi$}. On the other hand, this typical behavior does not tell us anything about the following question that arises naturally when we switch the quantification.

\begin{qn}\label{qn-except}
Given $\om \in \Omega$,  are there exceptional directions? What is their cardinality?
\end{qn}

By definition,  exceptional directions are necessarily random, i.e.\ they depend on $\omega$. From a geometric point of view, typical behavior is `local' (where $\xi$ is fixed), whereas  exceptional behavior is global, in the sense that all points of $\pG$ need to be considered. We next define the multiplicity of exceptional directions. 

\begin{defn}\label{def-mult}
Let $\{[a_i, \xi')_\omega: i \in I\}$ be a maximal collection of disjoint geodesic rays (either in $F(1,\omega)$ or in $T(\xi,\omega)$) all converging to $\xi'$. Then the cardinality of $I$ is called the multiplicity of $\xi'$. (Here, the collection is maximal in terms of its cardinality.)
\end{defn}

The following theorem summarizes the conclusions of \cite{BM24} about exceptional directions.

\begin{theorem}\label{thm-omni-BM24} The following hold for both forward and backward random trees for FPP on a hyperbolic group. 

	\begin{enumerate}
		\item If $G$ is not free (up to finite index), almost surely exceptional directions in $\pG$ exist and are dense in $\pG$.
		\item If $\pG$ has topological dimension $\dim_t \pG$ greater than 1, then almost surely there are 
		\emph{uncountably many} exceptional directions.
		\item If $G$ is a cocompact lattice in the hyperbolic plane, and the Cayley graph $\Gamma$ is planar, then  almost surely there are \emph{countably many} exceptional directions. 
		\item Let $G$ be a  hyperbolic group such that $\dim_t \pG = n-1$. Then
		there exists an exceptional direction $z \in \pG$ with multiplicity \emph{at least $n$}. Thus there exist hyperbolic groups with arbitrarily large
		multiplicity of an exceptional direction.
		\item Exceptional directions have measure zero with respect to the Patterson-Sullivan measure. In fact the set of exceptional directions has Hausdorff dimension (almost surely uniformly) smaller than that of $\pG$.
		\item For any (fixed) $\Gamma$, there exists an almost sure uniform upper bound 
		on the  multiplicity of an exceptional direction.
	\end{enumerate} 
\end{theorem}

\section{Distribution of coalescence point}\label{sec-coalescedistrbn}

The purpose of this section is to study the point of coalescence of two FPP geodesic rays going in the same direction. Recall that

\begin{itemize}
	\item  $\partial^2 G:= \{(\xi, \xi'): \xi, \xi' \in \pG; \, \xi\neq \xi'\}.$
	\item  {$\partial^3 G:= \{(\xi, \xi',\xi''): \xi, \xi',\xi'' \in \pG; \, \xi\neq \xi'\neq \xi'' \ne \xi \}.$}
\end{itemize}

Let $(\xi_1,  \xi_2, \xi_3) \in \partial^3 G $. The union of bi-infinite geodesics $(\xi_1,\xi_2) \cup  (\xi_2,\xi_3) \cup (\xi_3,\xi_1)$ in $\Gamma$ is called  a \emph{geodesic ideal triangle} or simply
an \emph{ideal triangle} with vertices at $\xi_1  \xi_2, \xi_3$ and is denoted as
$ \Delta\xi_1  \xi_2\xi_3$. There exist constants $D,D'$ (depending only on the hyperbolicity constant $\delta$ of $\Gamma$) such that the following holds:
\begin{enumerate}
\item There exists a point $c_{\xi_1  \xi_2, \xi_3} \in \Gamma$ lying in a $D-$neighborhood of each of the bi-infinite geodesics $(\xi_1,\xi_2),   (\xi_2,\xi_3), (\xi_3,\xi_1)$ (see \cite[Lemma 2.7]{mahan-sardar} for instance).
\item For any point $z  \in \Gamma$ lying in a $D-$neighborhood of each of the bi-infinite geodesics $(\xi_1,\xi_2),   (\xi_2,\xi_3), (\xi_3,\xi_1)$, $d(z, c_{\xi_1  \xi_2, \xi_3}) \leq D'$ (see the paragraph following \cite[Lemma 2.7]{mahan-sardar} for instance).
\end{enumerate}
We refer to $c_{\xi_1  \xi_2, \xi_3}$ as a \emph{deterministic centroid}, or simply a \emph{centroid} of $ \Delta\xi_1  \xi_2\xi_3$ (in the literature, $c_{\xi_1  \xi_2, \xi_3}$ is sometimes called the \emph{barycenter} of
$(\xi_1  \xi_2, \xi_3) \in \partial^3 G $). The above discussion shows that  $c_{\xi_1  \xi_2, \xi_3}$  is coarsely well-defined. We shall therefore refer to $c_{\xi_1  \xi_2, \xi_3}$ as the \emph{centroid} of $ \Delta\xi_1  \xi_2\xi_3$.

In the context of a random backward tree $T(\xi,\omega)$, the point
$\xi$ is distinguished. To make the notation suggestive we shall denote other ideal end-points of $T(\xi,\omega)$ by letters of the English alphabet.
 For any
{$a\neq b \neq \xi \ne a \in \partial G$} Theorem~\ref{thm-btree} ensures that $T(\xi,\omega)$ contains a subtree 
$T(a,b,\xi,\omega)$ given by the union of the random bi-infinite geodesics
$(a,\xi)_\omega$ and $(b,\xi)_\omega$. Thus, $T(a,b,\xi,\omega)$ is a tripod having a well-defined centroid $c_{ab\xi}(\omega)$. Note that 
$c_{ab\xi}(\omega)$ is the point of coalescence of $[o_1, \xi)_\om$ and $[o_2, \xi)_\om$ where $o_1, o_2$ are points in $(a,\xi)$ and $(b,\xi)$ close to $a, b$ respectively in the topology of the Gromov compactification $\Gamma \cup \pG$. We refer to $c_{ab\xi}(\omega)$ as the \emph{random centroid}
of $a,b,\xi$.

The aim of this section is to show that the distance of the  random centroid from the deterministic centroid has exponentially decaying tails under Assumption~\ref{assume-subexp}. To make this precise, we make the following definition. 

\begin{defn}
	Define the {\bf coalescence-centroid distance} to be a random variable $X_{ab\xi}$, measuring the distance from the point of coalescence to a centroid as follows:\\
	We choose a centroid $c_{ab\xi}$ of the triangle with vertices $a, b, \xi$.	Also, let $c_{ab\xi}(\om)$ denote the random centroid. By Theorem \ref{bmthmcoalesce}, 
	$c_{ab\xi}(\om) \in \Ga$ for a.e.\ $\om \in \Omega$. Define
	$$X_{ab\xi}(\om) = d(c_{ab\xi}, c_{ab\xi}(\om))$$ to be  the graph distance from the centroid  to random centroid.
\end{defn}

The main theorem of this section is the following. 

\begin{theorem}\label{thm-centroidexpdecay}
There exists $c>0, N_0 \in \natls$ such that for all $(a, b, \xi)\in \partial^3G$, and $N \geq N_0$, $\P [X_{ab\xi} \geq N] \leq e^{-cN}.$
\end{theorem}

\subsection{Background on coalescence in hyperbolic spaces}\label{sec-backcoalesce}

We refer the reader to \cite[Section 7.1]{BM19} for basics about hyperplanes in our context.
Let $[1,\xi)=\{1=x_0,x_1,\ldots,\}$ denote a geodesic ray in the direction $\xi\in \pG$. Let $i\geq 0,D\in  \natls$ and let $H_{iD}$ denote the hyperplane perpendicular to $[1,\xi)$ at $x_{iD}$. Then for $D$ sufficiently large, any $o\in \Ga$, and for a.e. $\omega \in \omegap$, every $\omega$-geodesic ray from $o$ to $\xi$ must cross $H_{iD}$ for all sufficiently large $i$.
To prove Theorem~\ref{thm-centroidexpdecay}, we first recall a crucial fact from  \cite{BM19}. The region between $H_{iD}$ and $H_{(i+1)D}$ is denoted by $\SSS(i,D)$ and is referred to as the $i-$th \emph{slab} in the direction $\xi$ at \emph{scale $D$}.
\begin{prop}\cite[Proposition 7.10]{BM19}
	\label{bmproppositive}
	For $D=D(\rho,\Ga)$ sufficiently large, there exists $\beta=\beta(D,\Ga, \rho)>0$ independent of $i$ such that for each $i\in \N$ there exists an event $B_{i}$ depending only on the configuration restricted to $N_{D/10}( \SSS(i,D))$ with the following properties:
	\begin{enumerate}
		\item $\P(B_{i}) \geq \beta.$
		\item {For every $o_2\in G$, for all sufficiently large $i$, on $B_{i}$, every pair of $\omega$-geodesics rays started from the two starting points $1$ and $o_2$ in the direction $\xi$ intersect on $[x_{iD},x_{(i+1)D}]$.} 
	\end{enumerate}
\end{prop}

\begin{rmk}\label{rmk-hyp2hyp}
The proof of \cite[Proposition 7.10]{BM19} in fact gives an explicit description of the event $B$. After constructing uniformly quasiconvex hyperplanes $\HH_i$ through $x_{iD}$ as above, $B_i$ denotes the event  that any pair of $\omega$ geodesics from $\HH_i$ to $\HH_{i+1}$ intersect on the geodesic segment $[x_{iD},x_{(i+1)D}]$ for 
$\omega \in B_i$.
\end{rmk}

The proof of Proposition \ref{bmproppositive} (in particular Remark~\ref{rmk-hyp2hyp}) gives the following more
effective version of Theorem \ref{bmthmcoalesce}. Abusing terminology slightly, we shall refer to triangles with \emph{at least one vertex on $\pG$} as ideal triangles below.

\begin{theorem}\label{bmthmcoalesceeff} There exists $k > 0$,
	a scale $D \geq 1$, $p \in (0,1)$, $N_0 \in \natls$ and $C\geq 1$ such that 
	for any given direction $\xi$, 
	  the following hold for all $o_1, o_2 \in G$:
	\begin{enumerate}
	\item Let $c=c(o_1,o_2, \xi)$ denote a centroid of the ideal triangle with vertices $o_1,o_2, \xi$, and let $[c,\xi)$ be a geodesic ray from $c $
	to $\xi$. Let $c=c_0, c_1, c_2, \dots$ be points on
	$[c,\xi)$ such that $d(c_i, c_0) = Di$.
	Then the probability that $[o_1,\xi)_\om$ and 
	$[o_2,\xi)_\om$  coalesce after $N_k([c,c_n])$  
	(i.e.\ $[o_1,\xi)_\om\cap [o_2,\xi)_\om \cap N_k([c,c_n]) =\emptyset$) 
 is at most $Cp^n$
	for $n \geq N_0$.
	\item 	{Let $H_i$ denote the hyperplane at $c_i$ perpendicular to $[c,\xi)$}. The probability that $[o_1,\xi)_\om$ has 2 disjoint subpaths between hyperplane {$H_i$ and $H_{i+M}$} (i.e.\  $[o_1,\xi)_\om$  as a parametrized path goes from $H_i$ to $H_{i+M}$ and then  backtracks   to $H_i$)  
	has probability  at most $Cp^M$ for $M \geq N_0$.
	\end{enumerate}

\end{theorem}

\begin{proof}
We briefly enunciate the ingredients necessary from \cite{BM19} to prove these statements. {Since the arguments are quite similar to the ones there, we omit the details}. Note that both statements in the Theorem follow from finding a lower bound on the probability of the event that there exists two disjoint $\omega-$geodesics from  $H_i$ to $H_{i+M}$. This is because the two
{disjoint} subpaths in item (2) are {disjoint} $\omega-$geodesics from {some point in 
$H_i$ to some point $H_{i+M}$ (the endpoints are not necessarily the same).}

Now, Proposition~\ref{bmproppositive} and Remark~\ref{rmk-hyp2hyp} give events (one for each $i$) of uniform
(i.e.\ independent of $i$)
positive probability $\beta>0$ on which all $\omega-$geodesics  between $H_i$ and $H_{i+1}$ intersect. {These events are not disjoint since the hyperplanes are only coarsely-well defined. However, by choosing $D$ sufficiently large it can be ensured that these events are $1$-dependent (in $i$), i.e., the collection of the events with even indices (as well as the ones with odd indices) are independent. The bounds now follow by choosing $p\in (1-\beta,1)$ appropriately large.}
\end{proof}

\subsection{Proof of Theorem~\ref{thm-centroidexpdecay}}
Let $a, b, \xi$ be as in the statement of 
Theorem~\ref{thm-centroidexpdecay}. Let $y=c_{ab\xi}$ denote the (coarsely well-defined deterministic) centroid, and $z=z(\omega)=c_{ab\xi}(\omega)$ the random centroid, i.e.\ the point at which $(a,\xi)_\omega$ and  $(b,\xi)_\omega$ coalesce. {For any sequences
$a_n \to a$, $b_n \to b$ 
we can assume without loss of generality that $d(a_n,y), d(b_n,y) \gg N$ for all $n$ large enough
(where $N$ is as in the statement of 
Theorem~\ref{thm-centroidexpdecay})}. Let $\TT= \TT(a,b,\xi) $ denote the tripod with vertices at $a, b, \xi$ and centroid at $y$.

\begin{figure}[htbp!]
\includegraphics[width=0.8\textwidth]{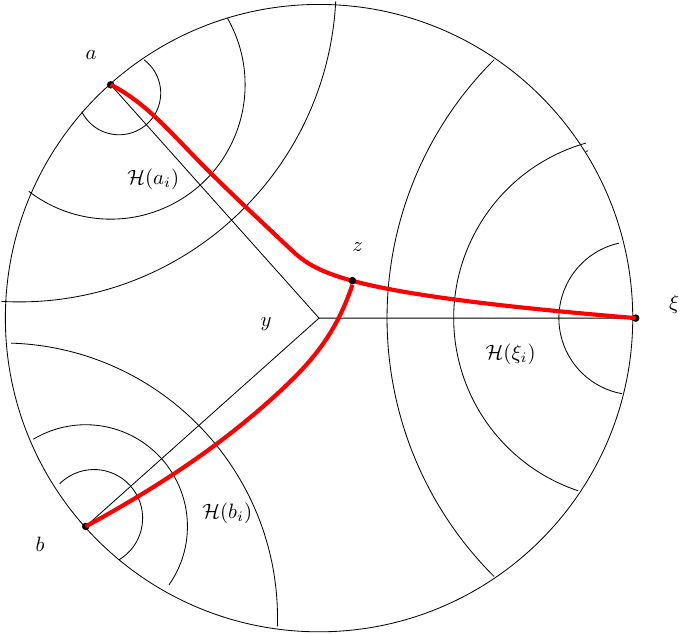}
\caption{Proof of Theorem \ref{thm-centroidexpdecay}: We want to show that the coalescence point $z$ of the geodesics $(a,\xi)_{\omega}$ and $(b,\xi)_{\omega}$ cannot be too far from the deterministic centroid $y$. By considering hyperplanes normal to $[y,a), [y,b)$ and $[y,\xi)$ and using Theorem \ref{bmthmcoalesceeff} we know that $z$ is unlikely to lie beyond too many of these hyperplanes from $y$ in any of the three directions. That $z$ is typically close to $y$ is then deduced using Proposition \ref{prop-effbt}.}
\label{f:tripod}
\end{figure}


Since $\Gamma$ is $\delta-$hyperbolic, there exists $k_0(=k_0(\delta))$ such that hyperplanes normal to distinct geodesic components of $\TT \setminus B_{k_0}(y)$ are disjoint.
Let $a_0, b_0, \xi_0$ be points on $[y,a)$, $[y,b)$, $[y,\xi)$ respectively at distance $k_0$ from $y$. Next, let $a_i, b_i, \xi_i$, $i=1,2,\cdots$ be points on $[y,a)$, $[y,b)$, $[y,\xi)$ respectively at distance $k_0+iD$ from $y$, where $D$ is the scale in Theorem~\ref{bmthmcoalesceeff}. {Thus, $a_i \to a, b_i \to b$, so that for $i$ large enough, $d(a_i,y), d(b_i,y) \gg N$ as in the first paragraph of the proof.}
Let $\HH(a_i), \HH(b_i), \HH(\xi_i)$ be hyperplanes through $a_i, b_i, \xi_i$ normal to $[y,a)$, $[y,b)$, $[y,\xi)$ respectively for $i=0,1,2,\cdots$.

Let $\ep \in (0,1)$ (to be chosen later) and $N$ be as in the statement of Theorem~\ref{thm-centroidexpdecay}. Then, by the first conclusion of Theorem~\ref{bmthmcoalesceeff}, we have the following:
\begin{enumerate}
\item The probability that $(a,\xi)_\omega$ and $(b,\xi)_\omega$ coalesce beyond 
$\HH(\xi_i)$ for $i=\ep N$ is at most $p^{\ep N}$. (Here and below, by "beyond" we shall mean the half space in $\Gamma \setminus  \HH(\xi_i)$ not containing the  centroid $y$.)
\item The probability that $(\xi,a)_\omega$ and $(b,a)_\omega$ coalesce 
at $a'$ beyond 
$\HH(a_i)$ for {$i=\floor{\frac{\ep N}{D}}$} is at most $p^{\ep N}$.
\item The probability that $(\xi,b)_\omega$ and $(a,b)_\omega$ coalesce at $b'$ beyond 
$\HH(b_i)$ for $i=\ep N$ is at most $p^{\ep N}$.
\end{enumerate} 
If  any of the  three events described in the items above happen, we immediately have the exponential decay demanded by Theorem~\ref{thm-centroidexpdecay}.

Hence, we can now assume that both the coalescence points $a'$ and $b'$ have nearest-point projections onto $[y,a)$ and $[y,b)$ respectively at distance at most
$\ep N$ from $y$.

Next, if the nearest-point-projection of $z$ onto the (uniformly quasiconvex) tripod $\TT$ lies on $(y,a)$ at distance greater than $2 \ep N$, then it follows that 
$(a,\xi)_\omega$ backtracks from the hyperplane $\HH(a_{\ep N})$ to  $\HH(a_{2\ep N})$, an event occurring with probability at most $p^{\ep N}$ by the second 
conclusion of Theorem~\ref{bmthmcoalesceeff}. 
{A similar argument works with $(y,b)$ or $(y,\xi)$ replacing $(y,a)$.} Again, we  have the exponential decay demanded by Theorem~\ref{thm-centroidexpdecay}.

Thus, we may assume without loss of generality that
\begin{enumerate}
\item the nearest point-projection of $z$ onto $(ay\xi)$ lies on $(ay]$ at distance
at most $m$ from $y$,
\item  the nearest point-projection of $z$ onto $(by\xi)$ lies on $(by]$ at distance
at most $m$ from $y$,
\item $m<2\ep N$.
\end{enumerate}

Let $z_a$ denote the nearest-point projection of $z$ onto $(ay\xi)$ so that 
{$d(z, z_a) \geq N-m\geq N(1-2\ep)$.} Similarly, let $z_b$ denote the nearest-point projection of $z$ onto $(by\xi)$ so that {$d(z, z_b) \geq N-m\geq N(1-2\ep)$.} In particular,
$[a,\xi)_{\omega}$ passes through $z$ at distance at least {$N(1-2\ep)$} from $z_a$.
{By Proposition~\ref{lem-effbt}}, this has probability at most {$e^{-c(1-2\ep)N}$.} Taking a union bound over  
{$m<2\ep N$} (i.e.\ the nearest-point projection $z_a$ is at distance at most  $2\ep N$ from $y$), we finally obtain an upper bound of
{$2 \ep N e^{-c(1-2\ep)N}$} on the event $d(z,y)\geq N$. Finally, choosing
$\ep$ small enough, we obtain an  upper bound of $e^{-cN/2}$. This completes the proof of Theorem~\ref{thm-centroidexpdecay}. \hfil $\Box$

\subsection{Cross-ratios and $n-$point functions}\label{sec-npoint} 
Ideal triangles, and hence approximating tripods are naturally associated
to 3 points in $\pG$. It is natural to wonder what the analog is for $n$ points on $\pG$. This subsection is devoted to this generalization. 
We would thus like to extend Theorem~\ref{thm-centroidexpdecay} to $n-$point functions. We start by investigating the natural $4-$point function
given by the hyperbolic cross-ratio.

We recall the notion of a hyperbolic cross-ratio (see for instance \cite{paulin-bdy}):
\begin{defn}
	The {\bf hyperbolic cross-ratio} of an ordered tuple of four points $x, y, z, w$ in $\Gamma$ is defined to be $$[x:y:z:w] := \frac{1}{2}
	d(x,w) + d(y,z) - d(x,y) - d(z,w).$$
	
	For $a, b, c, d \in \Gamma\cup \pG$, the  {\bf hyperbolic cross-ratio} is defined as
	$$[a:b:c:d] := \sup_{x_i \to a, y_i\to b, z_i\to c, w_i \to d} \liminf
	[x_i:y_i:z_i:w_i].$$
\end{defn}
Assume, as usual that $\Gamma$ is $\delta-$hyperbolic. Then,
 up to an additive constant $C_0$ depending only on $\delta$, $[x:y:z:w]$ equals 
 \begin{enumerate}
 \item $d([x,y],[z,w])$ if $d([x,y],[z,w])\geq C_0$
 \item $-d([x,w],[y,z])$ if $d([x,w],[y,z])\geq C_0$
 \item is bounded in absolute value by a constant $C_1$ depending only on $\delta$, otherwise.
 \end{enumerate}
 Here, as usual, $[x,y]$ denotes the geodesic segment between $x,y$.
For $a, b, c, d \in \Gamma\cup \pG$,  the same assertions go through, except that geodesic segments may be replaced by semi-infinite or bi-infinite geodesics. We refer the reader to \cite{paulin-bdy} for these and other basic properties of the cross-ratio. Next, fix a direction $\xi$ on the boundary.
The following Lemma now follows from these properties  of the cross-ratio.

\begin{lemma}\label{lem-xratio} Let $\pi_\xi$ denote a nearest-point projection onto $[1,\xi)$.
	
	If $[1:\xi:a: b] \gg 0$, then 	$d([a,b], [1,\xi))(\gg 0) $
	equals $|[1:\xi:a: b]|$ up to the uniformly bounded
	constant $C_0$ above. In particular, $d(\pi_\xi(a), \pi_\xi(b))\leq C_0$.
	
	If $[ 1: \xi: a: b] \ll 0$, then $d(\pi_{\xi}(a), \pi_{\xi}(b)) (\gg 0)$
	equals  $|[1:\xi:a: b]|$ up to the uniformly bounded
	constant $C_0$ above.
\end{lemma}

\begin{prop}\label{prop-4corr}
	If {$[ 1: \xi: a: b]\ll 0$, }then $c_{a1\xi}(\omega)$ and $X_{b1\xi}(\omega)$
	are {nearly independent}.
	If {$[ 1: \xi: a: b] \gg 0$, }  then $c_{a1\xi}(\omega)$ and $c_{b1\xi}(\omega)$
	are nearly the same.
	
	More precisely, there exists $C\geq 1, \, c\in (0,1)$ such that 
	\begin{enumerate}
		\item in the first case, {the total variation distance between the joint distribution of $(c_{a1\xi}(\omega),c_{a1\xi}(\omega))$ and the product of the two marginals is smaller than $\exp(c'\langle 1,\xi,a, b \rangle)$.} 
		\item  in the second case, there exists $\kappa > 0$ such that
		 $\P[{d(c_{a1\xi}(\omega), c_{b1\xi}(\omega))}\geq N]\leq e^{-\kappa N}$. 
	\end{enumerate}
\end{prop}

\begin{proof} 
{ For the first case, let $d(\pi_{\xi}(a), \pi_{\xi}(b)) \sim |[1:\xi:a: b]| =N$, say.

\begin{figure}
    \centering
    \includegraphics[width=0.8\linewidth]{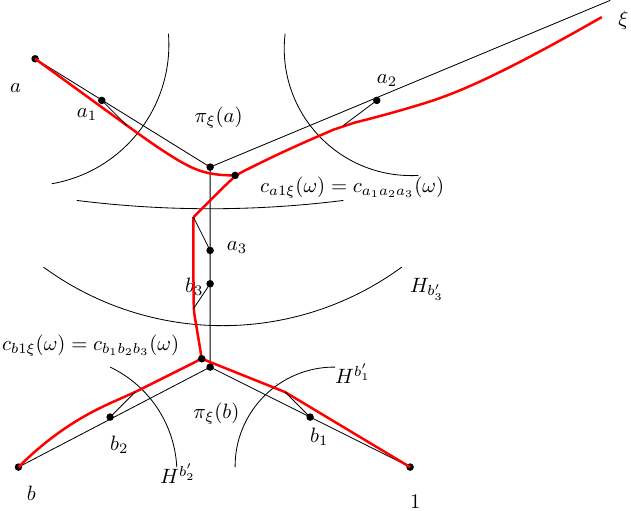}
    \caption{Proof of the first case in Proposition \ref{prop-4corr}: the points $b_1,b_2,b_3$ are at distance $N/4$ from $\pi_{\xi}(b)$ on $[1,\pi_{\xi}(b)], [b,\pi_{\xi}(b)]$ and $[\pi_{\xi}(a), \pi_{\xi}(b)]$ respectively. Points $b_i'$ (not marked in the figure) are midpoints of $[b_i,\pi_{\xi}(b)]$ and $H^{b'_i}$ are the hyperplanes passing through $b'_i$ perpendicular to these line segments. We show that with high probability the coalescence points $c_{b1\xi}(\omega)$ equals a random point $c_{b_1b_2b_3}(\omega)$ which depends only on the edge lengths near the points $\pi_{\xi}(b)$. A similar proxy $c_{a_1a_2a_3}(\omega)$ is constructed for $c_{a1\xi}(\omega)$. The near-independence follows since $\pi_{\xi}(a)$ and $\pi_{\xi}(b)$ are far away one another.}
    \label{f:l2g}
\end{figure}

\noindent{\bf{A local-to-global argument:}} see Figure \ref{f:l2g} for an illustration of this argument.  Consider the points $b_1,b_2,b_3$ at distance $N/4$ from $\pi_{\xi}(b)$ on $[1,\pi_{\xi}(b)], [b,\pi_{\xi}(b)]$ and $[\pi_{\xi}(a), \pi_{\xi}(b)]$ respectively. Consider hyperplanes $H^{b_1}, H^{b_2}, H^{b_3}$ passing through $b_1, b_2,b_3$ and perpendicular to $[1,\pi_{\xi}(b)], [b,\pi_{\xi}(b)]$ and $[\pi_{\xi}(a), \pi_{\xi}(b)]$ respectively. Consider the random "geodesic" $\gamma_{b_1b_3}$ between $b_1$ and $b_3$ where the minimum is taken over all paths contained in the region between $H^b_1$ and $H^{b_3}$. For $i=1,2,3$, let $b'_i$ denote the midpoint of $[b_i,\pi_{\xi}(b)]$. and let $H^{b'_i}$ denote the hyperplane through $b'_i$ perpendicular to $[b_i,b'_i]$. The coalescence argument in Theorem \ref{bmthmcoalesceeff} implies that on an event of probability $\ge 1-e^{-cN}$ we have that $[b,\xi)_{\omega}$ and $\gamma_{b_1b_3}$ agree in the region between $H^{b'_1}$ and $H^{b'_3}$. Constructing $\gamma_{b_2b_3}$ similarly and denoting the coalescence point of $\gamma_{b_1b_3}$ and $\gamma_{b_2b_3}$ by $c_{b_1b_2b_3}(\omega)$ it follows from Theorem \ref{thm-centroidexpdecay} that on an event of probability $\ge 1-e^{-cN}$, $c_{b_1b_2b_3}(\omega)=c_{b1\xi}(\omega)$. The point of this construction is that the random point $c_{b_1b_2b_3}(\omega)$ depends only on the local randomness around $\pi_{\xi}(b)$; more precisely, it only depends on the edge weights in the union of the region between $H^{b_1}$ and $H^{b_3}$ and the region between $H^{b_2}$ and $H^{b_3}$.
An essentially identical construction around the point $\pi_{\xi}(a)$ yields the existence of a random point $c_{a_1a_2a_3}(\omega)$ such that on an event of probability $\ge 1-e^{-cN}$, $c_{a_1a_2a_3}(\omega)=c_{a1\xi}(\omega)$ and $c_{a_1a_2a_3}(\omega)$ is independent of $c_{b_1b_2b_3}(\omega)$ (we used $d(\pi_{\xi}(a), \pi_{\xi}(b))=N$ and the fact that edge weights in disjoint regions are independent for the final conclusion). This completes the proof of the first case of the theorem.}


For the second case,  $d(c_{a1\xi},c_{b1\xi}) \leq c$. This is because 
\begin{enumerate}
\item $d(\pi_\xi(a), c_{a1\xi})$ and $d(\pi_\xi(b), c_{b1\xi})$ are uniformly bounded in terms of $\delta$ (a general fact about hyperbolic spaces), and
\item $d(\pi_\xi(a), \pi_\xi(b))\leq C_0$ by Lemma~\ref{lem-xratio}.
\end{enumerate} 
By Theorem~\ref{thm-centroidexpdecay}, $X_{a1\xi}\leq N/2$ with probability
at least $(1-e^{-\kappa N})$, i.e. $d(c_{a1\xi},c_{a1\xi}(\omega))\leq N/2$ with probability
at least $(1-e^{-\kappa N})$. The same is
true for  $d(c_{b1\xi},c_{b1\xi}(\omega))$. Hence, there exists $C_1=C_1(\delta)$ such that $\P[d(c_{a1\xi}(\omega), c_{b1\xi}(\omega))\geq N+C_1] \leq e^{-\kappa N} $. Choosing a slightly smaller $\kappa$ furnishes the conclusion.
\end{proof}


It is now easy to generalize Proposition~\ref{prop-4corr} to $(n+1)-$point functions. Let $a_1,\cdots, a_n, \xi$ be distinct points on $\pG$.
Let $\WW\HH(a_1,\cdots, a_n, \xi)$ denote the weak hull of $a_1,\cdots, a_n, \xi$, i.e.\ $\WW\HH(a_1,\cdots, a_n, \xi)$ consists of the union of bi-infinite
geodesics $(u,v)$ for $u, v \in \{a_1,\cdots, a_n, \xi\}$. There exists $C_1$ depending only on $\delta$ such that $\WW\HH(a_1,\cdots, a_n, \xi)$ is $C_1-$quasiconvex. Note that $C_1$ is independent of $n$.
Further, there exists a tree $\TT(a_1,\cdots, a_n, \xi)\subset \Gamma$ with ideal vertices
at $a_1,\cdots, a_n, \xi$ such that $\TT(a_1,\cdots, a_n, \xi)$ is a $\delta \log_2 n$ approximation of  $\WW\HH(a_1,\cdots, a_n, \xi)$ \cite[Chapter 2]{GhH}, i.e.\ 
\begin{enumerate}
\item $\WW\HH(a_1,\cdots, a_n, \xi)$ and $\TT(a_1,\cdots, a_n, \xi)$ lie within
a $\delta \log_2 n$ neighborhood of each other
\item the nearest-point projection of $\WW\HH(a_1,\cdots, a_n, \xi)$ onto $\TT(a_1,\cdots, a_n, \xi)$ has pre-images of points having diameter bounded above by  $\delta \log_2 n$.
\end{enumerate}

{We equip $\TT(a_1,\cdots, a_n, \xi)$ with a new simplicial structure, where $a_1,\cdots, a_n, \xi$ are the ideal vertices. Additional vertices $z$ in the 
interior of $\TT(a_1,\cdots, a_n, \xi)$ are introduced, if and only there are three geodesic paths with disjoint interiors of the
form $[z, p_j)$,  with $p_j \in \{a_1,\cdots, a_n, \xi\}$, $ j=1,2,3$.} Now, let $a, b$ be any two adjacent interior vertices of $\TT(a_1,\cdots, a_n, \xi)$. By Theorem~\ref{thm-btree}, for a.e.\ $\omega$, there exists a subtree
$\TT(a_1,\cdots, a_n, \xi)(\omega)$ of $T(\xi,\omega)$ with ideal vertices 
at $a_1,\cdots, a_n, \xi$. Let $a(\omega), b(\omega)$ denote the points on 
$\TT(a_1,\cdots, a_n, \xi)(\omega)$ corresponding to $a, b$. Here we necessarily
assume that $d(a,b)\gg 0$ so that such a correspondence is well-defined with probability close to 1.
Then a reprisal of the argument in  Proposition~\ref{prop-4corr} furnishes the following consequence:

\begin{cor}\label{cor-npoint} There exists $\kappa > 0$ such that the following holds.
Let $a_1,\cdots, a_n, \xi$ be distinct points in $\pG$ and $a, b$ adjacent interior vertices of $\TT(a_1,\cdots, a_n, \xi)$, such that $d(a,b)\gg 0$. Then
\begin{enumerate}
\item $\P[d(a(\omega), b(\omega))\geq d(a,b) + N]\leq e^{-\kappa N}$,
\item  $\P[d(a(\omega), b(\omega))\leq d(a,b) - N]\leq e^{-\kappa N}$
\end{enumerate}

\end{cor}

\begin{proof}
Without loss of generality we can assume that in the tree   $\TT(a_1,\cdots, a_n, \xi)$, the geodesic $[a,\xi)$ contains the vertex $b$. We can further assume without loss of generality that $[a,\xi)\subset (a_1,\xi)$. Since
$a, b$ are interior vertices, there exist ideal vertices ($a_2, a_3$ say), such that 
$(a_2,a] \cap  (a_1,\xi) = \{a\}$ and $(a_3,b] \cap  (a_1,\xi) = \{b\}$. 
{Hence $a, b$ are the unique internal vertices of the subtree 
   $\TT(a_1, a_2, a_3, \xi)$ of 
   the tree   $\TT(a_1,\cdots, a_n, \xi)$. To prove the Corollary, it suffices to consider 
   $\TT(a_1,\cdots, a_n, \xi)$.}
   
Thus we have reduced the problem to the case of a 4-point function with vertices $a_1,a_2,a_3,\xi$.
The Corollary now is a consequence of the second case of Proposition~\ref{prop-4corr} above.
{Indeed the second case of Proposition~\ref{prop-4corr} above guarantees that 
 $\P(d(a(\om), a) \leq N/2)$, $\P(d(b(\om), b) \leq N/2)$ have exponentially decaying tails (in $N$).
 Finally, a reprise of the local-to-global argument in Proposition~\ref{prop-4corr} for large $N/2$ balls around $a, b$ furnishes the conclusion. }
\end{proof}

\subsection{Proof of Proposition \ref{prop-bivel}}
\label{s:bivel}
As promised in the previous section, we shall now provide the proof of Proposition \ref{prop-bivel}.

\begin{proof}[Proof of Proposition \ref{prop-bivel}]
 Let $\partial_0 G \subset \partial G $ denote the full measure subset (with respect to the Patterson-Sullivan measure) for which
	the velocity exists in the sense of Theorem~\ref{thm-vexists} and is equal to the constant $v$. Replacing
	$\partial_0 G$ by $\bigcap_{g \in G} g.\partial_0 G$ if necessary, we may assume without loss of generality that $\partial_0 G$ is $G-$invariant. Let  $\partial^2_0 G =  (\partial_0 G \times  \partial_0 G) \cap \partial^2 G$. We claim that $\partial^2_0 G$ provides the full measure subset demanded by the Proposition.
	
	By $G-$invariance of $\partial^2_0 G$, we can assume without loss of generality that $ (x,y)$ passes through $1 \in G$. Now, as in \cite[Section 7.1]{BM19}, construct hyperplanes along $ (x,y)$ so that the two hyperplanes closest to $1$ cut $(x,y)$ at $D/2$ from 1. 
	
	Let $\Sigma_1(x,y)$ denote the collection of paths from $x$ to $y$
	passing through $1$. Then the velocity (as in Definition~\ref{def-vel} and \ref{def-bivel}) restricted to paths in $\Sigma_1(x,y)$ 
	equals $v$.  This is because we are simply
	concatenating paths from $1$ to $x$ and $1$ to $y$ in this case. Hence $v(x,y) \leq v$. We want to show that it cannot be strictly less.

	Let $x_n \in [1,x)$ be such that $d(x_n,1)=n$. Let $y_n \in [1,y)$ be such that $d(y_n,1)=n$.  It suffices to show that for any $\epsilon>0$,
	$$\liminf_{m,n\to \infty} \frac{\E T(x_{m},y_n)}{m+n}\ge v-\epsilon.$$
	
	Consider the event $A(M,N)$ that consists of all
	$\omega \in \Omega$ such that
	\begin{enumerate}
		\item the $\omega-$geodesic from $x_m$ to $y_n$ passes through $B_M(1)$,
		\item passage time on any edge contained in  $B_M(1)$ is bounded above by $N$.
	\end{enumerate}
	Notice next that, for $m,n\gg M$, on the event $A(M,N)$, by considering a point $v\in B_{M}(1)$ that is on the $\omega$-geodesic between $x_m$ and $y_n$ and considering a path from $1$ to $v$ contained in $B_M(1)$ we get 
	$$T(x_m,y_n)\ge T(1,x_m)+T(1,y_n)-2NM.$$
	It follows that
	\begin{eqnarray*}
	\E T(x_m,y_n) &\ge \E T(x_{m},y_{n})1_{A(M,N)} \ge\\  \E T(1,x_m)+\E T(1,y_n) &-2NM-\E [T(1,x_m)+T(1,y_n)]1_{A(M,N)^c}.
	\end{eqnarray*}
	{It follows from Theorem \ref{thm-vexists} and Theorem \ref{t:linvar} that 
	$\E [(T(1,x_m)+T(1,y_n))^2]=O((m+n)^2)$ and by Cauchy-Schwarz inequality, it follows that} 
	$$\E [T(1,x_m)+T(1,y_n)]1_{A(M,N)^c}=(m+n)O(\P(A(M,N)^c)^{1/2}).$$
	{Given $\epsilon$, we first choose $M$ sufficiently large such that Proposition \ref{prop-effbt} guarantees that the probability of the complement of the first event in the definition of $A(M,N)$ is $\ll \epsilon^2$. Now fix $N\gg M$ such that the probability of the complement of the second event in the definition of $A(M,N)$ is also $\ll \epsilon^2$. This is ensured by Assumption \ref{assume-subexp} and the fact that the number of edges in $B_M(1)$ is uniformly bounded as a function of $M$.} Together, these imply $\P(A(M,N)^c)^{1/2}\ll \epsilon$. The proof can now be completed by sending $m,n\to \infty$, and invoking Theorem \ref{thm-vexists}. 
%
%
%
%
%
%
\end{proof}

\section{Properties of the backward tree}
\label{sec-btree1}

In this section we shall study some properties of the backward geodesic tree $T(\xi,\omega)$ for a fixed $\xi$. In particular, we shall study the backward subtrees of $T(\xi,\omega)$ rooted at a given vertex $v\in \Ga$, i.e., the set of vertices $z$ such that $[z,\xi)_{\omega}$ passes through $v$.

\subsection{Backward trees starting from points on a horosphere}\label{sec-backcorr} 
Let $[1, \xi)$ be a geodesic ray, and $\LL$ denote the horosphere through $1$ based at $\xi$, i.e.\ $\LL$ is the level set of the Busemann function given by the condition $1 \in \LL$. Let $\DD_\xi$ denote the corresponding horoball.
Given any $a \in G= V (\Gamma)$ outside $\DD_\xi$ and a geodesic $[a,\xi)$, let {$\Pi(a) =[a,\xi) \cap \LL $}. By quasiconvexity of $\DD_\xi$, $\Pi(a)$ is coarsely well-defined; in fact, for any such $a$ and
any two geodesic rays, $[a,\xi)_1$, $[a,\xi)_2$, $d([a,\xi)_1 \cap \LL, [a,\xi)_2 \cap \LL)$ is at most $2\delta$ if $\Gamma$ is $\delta-$hyperbolic.
However, for $h \in \LL$, the collection of bi-infinite geodesics through $h$ forward asymptotic to $\xi$ is a well-defined set. 
{Hence, $\Pi^{-1}(h)$ is well-defined
(and not just coarsely defined), as it consists of all $a$ such that \emph{some} geodesic ray
$[a,\xi)$ passes through $h$.}

 Recall that $T(\xi,\omega)$ is the backward (random) tree.
 For all $h \in \LL$, let $T(\xi,h,\omega)$ denote the component of 
 $T(\xi,\omega)\setminus \mbox{Int}(\DD_\xi)$ containing $h$.
 The aim of this section is to explore quantitatively the relation between  backward trees 
 $T(\xi,h_1,\omega)$ and $T(\xi,h_2,\omega)$ with distance $d(h_1, h_2)$ for 
 $h_1, h_2 \in \LL$.
 We start with the following estimate.
 
 \begin{lemma}\label{lem-decorr-tree}
 Given $c>0$, there exists $D_0 > 0$ such that for $D>D_0$, the following holds. The probability that there exists
   $a \in \Pi^{-1}(h_1) $ such that $a \in T(\xi,1,\omega) $ is at most $e^{-cd(h_1, 1)}$ for $d(h_1, 1)\geq D_0$.
 \end{lemma}
 
 \begin{proof}
 Note first that by Gromov-hyperbolicity $d(1, [h_1,\xi)) > \frac{1}{2}d(h_1, 1) -2\delta $. Let $w \in  [h_1,\xi) $ denote a nearest-point projection of $1$, i.e.\ $d(1, [h_1,\xi)) = d(1,w)=D_1$, say. Thus, {$D_1  \geq \frac{1}{2}d(h_1, 1)-2\delta $}. Proposition~\ref{lem-effbt} now furnishes the conclusion.
 \end{proof}

\begin{rmk}
Note that the constant $c$ in Proposition~\ref{prop-effbt} and Lemma~\ref{lem-decorr-tree} can be made as large as we want at the cost of increasing the threshold distance $R_0$.
\end{rmk}

\begin{cor}\label{cor-decorr-tree}
Given $c>0$, there exists $D_0 > 0$ such that for $D>D_0$, the following holds. The probability
 that there exists $h_1$ with $a \in \Pi^{-1}(h_1) $, such that $a \in T(\xi,1,\omega) $ with $d(h_1 , 1) \geq D$ is at most $e^{-cD}$.
\end{cor}

\begin{proof}
	{Note first that $\{g \in \LL: d(g,1) = L \}$ has cardinality at most $C L^\alpha$ for some $C \geq 1, \alpha >0$ depending only on $\Gamma$ (we might as well choose $\alpha$ to equal the exponential growth rate of $\Gamma$ itself
    \cite{coornert-pjm}).}
The Corollary now follows from Lemma~\ref{lem-decorr-tree} by taking a union bound over
$d(h_1, 1) \in [D, \infty)$ after choosing $c$ (and hence $D_0$) large enough to start with.
\end{proof}

Let $d_\LL$ denote the intrinsic metric on the horosphere $\LL$. (Here, we assume that $G$ is one-ended; so that $\LL$  is coarsely connected, i.e.\ after possibly thickening $\LL$ slightly, the latter is connected. The metric $d_\LL$ is then the metric induced on $\LL$ from the \emph{path-metric} in the thickening.) Then we have the following.

{
\begin{lemma}\label{lem-polyg}
Let $G$ be a one-ended Gromov hyperbolic group. Let $\LL$ be a horosphere based at $\xi \in \pG$. We equip the vertex set of $\LL$ with a metric $d_\LL$, where $d_\LL$ is the metric induced from 
the path metric on a $10\delta-$thickening of $\LL$.
Then
$(\LL,d_\LL) $ has polynomial growth.
\end{lemma}

\begin{proof}
    The idea of the proof is contained in \cite[Lemma 3.16]{lmm}. The  basic idea is that
    horospheres are limits of spheres $\partial B_n(1):= \{x \in G: d(g,1)=n\}$ in a hyperbolic group, coupled with the fact that $\partial B_n(1)$ equipped with its intrinsic metric has polynomial growth.
\end{proof}}

Combining Corollary 5.3 and Lemma 5.4 gives the following superpolynomial decay estimate.



\begin{theorem}\label{thm-decorr-tree}
	Given $c>1$, there exists $D_0 > 0$ such that for $D>D_0$, the following holds. The probability
	that there exists $h_1$ with $a \in \Pi^{-1}(h_1) $, such that $a \in T(\xi,1,\omega) $ with $d_\LL(h_1 , 1) \geq D$ is at most $D^{-c}$.
\end{theorem}

{
\begin{proof}
    The only extra metric ingredient necessary is the well-known fact that horospheres are exponentially distorted, i.e.\ there exists $D_0 > 0, \kappa \geq 1$ such that for $D>D_0$, $d_\LL(h_1 , 1) \geq D$ implies
    $d(h_1 , 1) \geq \frac{1}{\kappa}\log D$.
\end{proof}}

\subsection{Growth of backward trees}\label{sec-lambda}
Let $ \LL_D(1)$ denote the $D-$neighborhood of $1$ in $(\LL,d_\LL)$. Let $\CC_D = \Pi^{-1}(\LL_D(1))$ (recall from Section~\ref{sec-backcorr} that $\Pi$ is the nearest point projection onto $\LL$).

 \subsubsection{Volume growth of $\CC_D$}
 We recall a basic theorem of Coornaert:
 
 \begin{theorem}\cite[Theorem 7.2]{coornert-pjm}\label{thm-coorn}
 There is a constant $C>1$ and $\lambda>1$ such that the cardinality $|B_n(1)|$ of the $n-$ball in $\Gamma$ satisfies
 $$\frac{1}{C} \lambda^n \leq |B_n(1)| \leq  C  \lambda^n,$$
 for all $n\in \natls$.
 \end{theorem}
 
 In fact, more is true. As noted by Calegari in \cite[Lemma 2.5.6]{calegari-notes}, the shadows $$S(g, R)=\{ \eta \in \partial G: d(g, [1,\eta)) \leq R, \, \forall \,  \text{geodesic rays}\, [1,\eta)\}$$  with $d(g,1) = n$ cover
$\partial G$ efficiently, i.e.\ any point in $\partial G$ lies in at most $N=N(R)$ shadows. It is important to note that $N(R)$ is independent of $n$ and $g$. 

A Theorem of Cannon \cite{cannon-conetype} \cite[Theorem 3.2.2]{calegari-notes}
shows that, after fixing an order on a finite generating set of $G$,
 the geodesic language 
 of lexicographically
first geodesics is prefix-closed and regular. In particular, it is the language accepted by a finite state automaton $\GG$, i.e.\ a topological Markov chain where all vertices are accept states.
 Further, Calegari notes (cf.\ \cite[Lemma 3.4.2]{calegari-notes}) that all but exponentially few paths $\gamma$ of length $n$ in $\Gamma$ are entirely
contained in one of the maximal components of $\GG$, except for a vanishingly small prefix and  suffix. 

Let $d(g,1) = n$ and $\eta \in S(g,R)$.
Let $\Theta (g,R) = \bigcup_{\eta \in S(g,R)} [1,\eta)$. Also, let $$\Theta (g,R,n,m) = \{v \in \Theta (g,R): n \leq d(v,1)\leq n+m \}.$$
Note that {$\Theta (g,R,n,0) = \{v \in \Theta (g,R):  d(v,1)=n \}.$}
Then, the fact that maximal components of $\GG$ determine growth along with 
 \cite[Lemma 3.4.2]{calegari-notes} gives us the following:

\begin{cor}\label{cor-coornrefined}
	There is a constant $C>1$ and $\lambda>1$ such that the cardinality $|\Theta (g,R,n,m)|$ of $\Theta (g,R,n,m)$ satisfies
	$$\frac{1}{C} |\Theta (g,R,n,n)| \lambda^m \leq |\Theta (g,R,n,m)| \leq  C 
	 |\Theta (g,R,n,n)| \lambda^m,$$
	for all $n, m\in \natls$.
\end{cor}

Note that the constant $\lambda$ in Theorem~\ref{cor-coornrefined} is the unique Perron-Frobenius eigenvalue of the finite state automaton $\GG$ encoding the geodesic language of $G$ \cite[Section 3.4]{calegari-notes}.
 
Returning to the horospherical setup, let   $\CC_D(R)$ denote the subset of $\CC_D$ consisting of vertices satisfying $d(v,\LL) \leq R$.
 
 \begin{prop}\label{prop-growth}
 There is a constant $C>1$ and $\lambda>1$ such that the cardinality $|\CC_D(R)|$ of $\CC_D(R)$ satisfies
 	$$\frac{1}{C} |\LL_D(1)| \lambda^R \leq |\CC_D(R)| \leq  C 
 |\LL_D(1)| \lambda^R.$$
 \end{prop}
 
 \begin{proof}
 We first note that the horoball $\DD_\xi$ is the limit of balls $B(\xi_n,n)$, where
 $\xi_n \in [1,\xi)$ satisfies $d(1,\xi_n)=n$. Since Corollary~\ref{cor-coornrefined} applies to all the balls $B(\xi_n,n)$ with fixed constants $C$ and $\lambda$, passing to the limit as $n\to \infty$ furnishes the Proposition.
 \end{proof}

 
 Let $\PP(D,R,\omega)$ denote the proportion of vertices $a$ in  $\CC_D(R)$ such that 
 $T(\xi,a,\omega) \cap \LL_D(1) = \emptyset$, i.e.\ it is the proportion of vertices in  $\CC_D(R)$ that have $\omega-$forward trajectories towards $\xi$ hitting $\LL$ outside $\LL_D(1)$. Let $\PP(D,R,\omega, \epsilon)$ denote the event that at most a fraction $\ep$ of vertices $a$ in  $\CC_D(R)$ satisfy 
 $T(\xi,a,\omega) \cap \LL_D(1) = \emptyset$.
 
 \begin{prop}\label{prop-proportion}
 	For any $\ep>0$, there exists $D_0$ such that for $D \geq D_0$, $R \geq 0$, $$\P(\PP(D,R,\omega, \epsilon)) > 1-\ep.$$ 
 \end{prop}
 
 \begin{proof}
 	Since $\LL$ is amenable by Lemma~\ref{lem-polyg}, the $(D-k)-$shell 
 	$$\LL(k,D,1):=\LL_D(1) \setminus \LL_{D-k}(1)$$ satisfies $\frac{|\LL(k,D,1)|}{|\LL_D(1)|} \to 0$ as 
    {$D \to \infty$ and $k=o(D)$. We make a choice for $k=o(D)$, e.g.\ $k=D^\alpha$ for some $\alpha \in (0,1)$) and $D$ large enough.}

 	Next, we  see that given $c>1$, there  exists $D_0 > 0$ such that the following holds.
    {
    For $D>D_0$, and $k=o(D)$ as above, 
 	the probability
 	that there exists $h \in \CC_{D-k}
    (=\Pi^{-1} (\LL_{D-k}(1))$, $a \in \LL$ satisfying the following
    \begin{enumerate}
        \item $h \in T(\xi,a,\omega) $,
        \item $d_\LL(a , 1) \geq D$
    \end{enumerate}
       is at most $D^{-c}$. 
Note that the existence of such an $h$ is equivalent to the statement that the backward tree
at $\xi$ contains $h$ in the backward branch from 
$a$, i.e.\ the $\omega-$geodesic from the `interior core point' $h$ to $\xi$ passes through $a\in \LL$ lying outside a $D-$ball 
(in the $d_\LL-$metric) about $1 \in \LL$.
       
       Indeed, this is proved by interchanging the roles of $1$ and $h_1$ in Theorem~\ref{thm-decorr-tree} and taking a union bound over
       $h \in \big(\CC_{D-k} (mR) \setminus \CC_{D-k} ((m-1)R)\big)$.  (Note that in Theorem~\ref{thm-decorr-tree} the roles played by $1, h_1$ are interchangeable and have nothing to do with the action of $G$.)
} The proposition now follows.
 	\end{proof}

 {
We shall now define the \emph{average growth rate} of the backward tree $T(\xi,\omega)$. By Lemma~\ref{lem-polyg}, horospheres $\LL$ in hyperbolic space exhibit polynomial growth, in particular they are amenable. Let $\LL_D$ denote the $D-$ball around a fixed base-point (say 1) in the 
horosphere $\LL$ with respect to its intrinsic metric $d_\LL$.
Let $\{\LL_k\}$ denote the sequence of $k-$balls. Then (due to polynomial growth of $\LL$), $\{\LL_k\}$ gives an exhausting F\"olner sequence for $\LL$, 
 i.e.\ $$\limsup_{k \to \infty} \frac{\int_{\LL_k}f}{\rm{vol}({\LL_k})})$$ is well-defined when $f$ is well-behaved (here, the integral is being taken with respect to the \emph{counting measure}).}
It remains to define the relevant function $f$. Recall that $T(\xi,h,\omega)$ denotes the component of $T(\xi,\omega)\setminus \mbox{Int}(\DD_\xi)$ containing $h \in \LL$.
Let $d(\xi,h,\omega)$ denote the intrinsic (graph) metric on $T(\xi,h,\omega)$, and let $T_n(\xi,h,\omega) \subset T(\xi,h,\omega)$ denote the finite subtree consisting of vertices at $d(\xi,h,\omega)-$distance at most $n$ from $h$. Also, let
$V_n(\xi,h,\omega) = |T_n(\xi,h,\omega)|$ denote the number of vertices in $T_n(\xi,h,\omega)$. 
Let $$f_n(h):=\E V_n(\xi,h,\omega) = \int_{\Omega} V_n(\xi,h,\omega)d\omega.$$
Next, let
{ $F(k,n) :=\frac{\int_{\LL_k}f_n}{\rm{vol}({\LL_k})})$
and $F(n) = \limsup_{k\to \infty} F(k,n)$.
\begin{defn}\label{def-avggr}
The \emph{average growth rate} of the backward tree $T(\xi,\omega)$ is defined to be $\limsup_{n \to \infty} \frac{1}{n} \log F(n)$.
\end{defn}}
Assembling the above pieces,  
 we have:
 
 \begin{theorem}\label{thm-backvalence}
 The average growth rate of the backward tree $T(\xi,\omega)$ is 
 {$\log\lambda$}. 
 \end{theorem}

 \begin{proof}
 By  Proposition~\ref{prop-growth}, 	$$\frac{1}{C} |\LL_D(1)| \lambda^R \leq |\CC_D(R)| \leq  C 
 |\LL_D(1)| \lambda^R.$$
 By Proposition~\ref{prop-proportion} the proportion of vertices $a$ in  $\CC_D(R)$ such that 
 $T(\xi,a,\omega) \cap \LL_D(1) = \emptyset$ is arbitrarily small.
 
 	Note also that  the probability
 that there exists $h \notin \CC_{D}$ such that $h \in T(\xi,a,\omega) $ with $d_\LL(a , 1) \leq D-k$ is at most $D^{-c}$. (To see this one directly applies Theorem~\ref{thm-decorr-tree} and takes a union bound over $D' \geq 2 D$.)
 
 The two above estimates coupled with amenability of $\LL$ (needed to extract a meaningful average as discussed above) prove the theorem.
 \end{proof}
 
 The notion of average growth rate given in Definition~\ref{def-avggr} is perhaps somewhat restrictive. Finer versions would be worth exploring.
 
\subsection{Finite backward trees and bubbles of positive curvature}\label{sec-bubble}
We have seen that there are vertices $v$ in the backward tree such that the backward subtree of $T(\xi,\omega)$ rooted at $v$ is infinite. In this subsection, we show that for one-ended groups, there is a positive probability that the backward subtree rooted at a given vertex is finite under an additional assumption.

In this subsection, we assume explicitly that the hyperbolic group $G$ is one-ended, equivalently $\pG$ is connected.
As before, the closed ball of radius $C$ about $x \in G$ will be denoted as $B_C(x)$, and its boundary, the sphere, by $S_C(x)$. The region between balls
 $$S_C^k(x) = B_{C+k}(x) \setminus B_C(x) $$ will be called the {\bf $k-$shell} about $B_C(x)$. {If $G$ is one-ended, there exists  $k>0$ such that for any $u, v \in S_C(x)$, there exists a path
joining $u, v$ lying in $ S_C^k(x)$
{\cite{bes-mess}.}}

{
\begin{defn}\label{def-bubble}
	A ball $B_C(x)$  is said to be a {\bf bubble of positive curvature} with respect to $\om \in \omegap$ if there exists $k>0$ such that for any $u, v \in S_C(x)$, the $\om-$geodesic 
	 joining $u, v$ lies in $S_C^k(x)$.
\end{defn}
Thus, in a bubble of positive curvature, the
unconstrained $\omega-$geodesic between boundary points of $B_{C}(x)$ lies in the shell, or equivalently,
shell paths are cheaper than paths through the ball for points on $S^{k}_C(x)$.}

The reason for the terminology in Definition~\ref{def-bubble} is as follows. Consider a sphere of radius 1. Remove an $\ep$ neighborhood $D_\ep$ of the North pole and glue the remainder to the boundary of a copy of $\R^2 \setminus D$, where $D\subset \R^2$ is a round disk such that
$\partial D_\ep$ and $\partial D$ have the same length. Smooth out the resulting 2-manifold by a small perturbation of the boundary circle. Then for all $\ep$ sufficiently small the resulting 2-manifold has a bubble of positive curvature away from the boundary of $S^2 \setminus D_\ep$.

\begin{defn}\label{def-deadend}
Let $\xi \in \pG$. We say that $a \in \Gamma$ is a dead-end of $T(\xi,\omega)$, if
$[a,\xi)_\omega$ is not a proper subset of any $\omega-$geodesic of the form $[b,\xi)_\omega \subset T(\xi,\omega)$.
\end{defn}


\begin{prop}\label{prop-bpc} Let $G$ be one-ended. {Assume also that, in addition to Assumption \ref{assume-subexp}, the edge length distribution has support containing an interval of the form $[0,h]$}. For every $R\gg 0$ sufficiently large, and $a \in G$, $a$ is the center of a bubble $B_R(a)$ of positive curvature {with positive probability}. Further, for all  $\xi\in \pG$,  there exists a
 dead-end of $T(\xi,\omega)$ with positive probability. Further, with positive probability the backward subtree rooted at a given vertex is finite. 
\end{prop}

\begin{proof}
	Fix $a \in G$ and $R>0$ be sufficiently large. Choose $c > 0$ such that $\rho ([c, \infty)) > 0$. Let $R_0 = R_0(R,k)$ be the total number of edges in $S_R^k(a)$, where $k$ is chosen such that for any $u, v \in S_R^k(a)$, there exists a path joining $u,v$ lying in $ S_R(a)$. {Recall that $k$ depends only on $G$ and its generating set as $G$ is one-ended \cite{bes-mess}. }
    By our hypothesis on $\rho$,  $\rho [\om(e) < \frac{c}{R_0} ] = \ep > 0$. 
	$${\P\bigg[ \om(e) < \frac{c}{R_0} \forall e \in S_R^k(a) \, \text{and} \, \om(e) > c \forall e \in B_R(a) \cup S^k_{R+k}(a)\bigg] > 0}.$$ 
    {In the above event, edges in $B_R(a)$ refer to all the edges with both endpoints in $B_R(a)$ and edges in $S_R^k(a)$ refer to all edges in $B_{R+k}(a)$ that are not in $B_R(a)$. We claim that such a distribution of weights makes 
	$B_R(a)$ a bubble of positive curvature. Indeed, on the above event, for any $u,v\in S_R(a)$ and any path $\gamma$ from $u$ to $v$ in $S_R^k(a)$, we have $\ell_{\omega}(\gamma)<c$ whereas for any path $\gamma$ from $u$ to $v$ not contained in $S_R^k(a)$ must have $\ell_{\omega}(\gamma)>c$ (since such a $\gamma$ must contain at least one edge in $B_R(a) \cup S^k_{R+k}(a)$).} 
    
    
	Let $P_R(a)$ denote the probability that $B_R(a)$ is a bubble of positive curvature. Then, by the above construction, $P_R(a)\ge \ep$ proving the first assertion.

	Let $\Omega_0 \subset \Omega$ be the collection of configurations for which 
	$B_R(a)$ is a bubble of positive curvature. We shall prove that for all $\omega \in \Omega_0$, and all $\xi \in \pG$,   there exists $x \in B_R(a)$ such that $x$ is a 
	 dead-end of $T(\xi,\omega)$. This will prove the second assertion. 
	 
	 Let  $[a,\xi)_\omega $ denote the $\omega-$geodesic ray from $a$ to $\xi$.
Since $S_C^k(a)$ separates $a, \xi$, there exists $u \in [a,\xi)_\omega \cap S_C^k(a)$.
We claim that if $[a,\xi)_\omega \subset [b,\xi)_\omega \subset T(\xi,\omega)$, then 
$b \in B_R(a)$. Else, there exists  $v \in [b,a]_\omega \cap S_C^k(a)$, and 
$[u,v]_\omega$ passes through $a$. In particular,  $[u,v]_\omega$ starts and ends in 
$S_C^k(a)$ and contains $a$. This contradicts the fact that 
	$B_R(a)$ is a bubble of positive curvature. 
    {From the above argument it also follows that the backward subtree rooted at the centre of a bubble of positive curvature is finite, proving the final assertion of the proposition.} 
\end{proof}

{ Since the action of $G$ on the configuration space $\Omega$ is ergodic, and for a fixed $R$ the event that there exists a bubble of positive curvature of radius $R$ is invariant under the translation of a configuration in $\omega$ by an element of $G$, we get the following corollary.
\begin{cor}
    \label{c:bubble}
    Under the hypothesis of Proposition \ref{prop-bpc}, for all sufficiently large $R$, almost surely there exists $x\in \Gamma$ such that $B_R(x)$ is a bubble of positive curvature. It also follows that  there are dead-ends in $T(\xi,\omega)$ for $\P$ a.e. $\omega$. 
\end{cor}

Choosing $R=n, n+1, n+2, ...$ for $n$ sufficiently large, it follows immediately from
Corollary~\ref{c:bubble} that for almost every
$\omega \in \Omega$, and all $n \in \natls$ sufficiently large, there exists $x_n \in \Gamma$ such that $B_n(x_n)$ is a bubble of positive curvature.
In particular, we have the following:

\begin{cor}\label{cor-nonhyp}
    Under the hypothesis of Proposition \ref{prop-bpc}, 
    for almost every
$\omega \in \Omega$, $(\Gamma, d_\om)$ is not hyperbolic.
\end{cor}

Corollary~\ref{cor-nonhyp} gives a different proof of a version of \cite[Theorem A]{bjq25}
under somewhat more restrictive hypotheses.
}

\subsection{Law of backward subtrees from random centroid}\label{sec-law}
As before, $x, y, z \in \partial^3G$. By Theorem~\ref{thm-btree}, there exists a full measure subset $\Omega(x,y,z) \subset \Omega$ such that $x,y,z$ are
\emph{non-exceptional directions} for $\omega \in \Omega(x,y,z)$ in the sense that there are unique bi-infinite geodesics $(x,z)_\omega$ and 
$(y,z)_\omega$ that coalesce.  Let $c(x,y,z, \omega)$ denote  the  random centroid for $\omega \in \Omega(x,y,z)$ (note that 
 $c(x,y,z, \omega)$ is well-defined whenever $x,y,z$ are
 non-exceptional directions).
Assuming that the triple is ordered in such a way that
$z$ is the point from which backward trees are considered, we have a random
backward tree from $c(x,y,z, \omega)$ giving a law $L(x,y,z)$ for backward trees from
the random centroid.

We start with an  observation of Gromov that asserts that a non-elementary hyperbolic group
$G$ acts properly discontinuously and cocompactly on $\partial^3G$. 
The converse is a theorem of Bowditch \cite{bowditch-jams}. Note that 
 $L(x,y,z)=L(g.x,g.y,g.z)$ for any $g \in G$.
Let $K=(\partial^3G)/G$.
Hence, there is a compact family of laws $L(x,y,z)$ parametrized by the quotient $K$.

\begin{defn} Let $z \in \partial G$, $z' \neq z  \in \partial G$.
	We shall say that an  \emph{asymptotic law} $L(z',z)$ exists if
	for all $ x_n \neq  y_n  \in \partial G$, with $ x_n  \to z'; \,  y_n
	 \to z'$, the laws $L(x_n,y_n,z)$ converge.
\end{defn} 

Bader and Furman \cite{bader-furman} prove the existence of a measure $\nu_{BMS}$ on
$\partial^2 G$ (similar to the Bowen-Margulis-Sullivan measure on the space of geodesics in a rank one symmetric space) absolutely continuous with respect to $\nu \otimes \nu$ on $\partial^2 G$ such that there exists $C \geq 1$ such that the Radon Nikodym derivatives
 $\frac{dg^*\nu_{BMS}}{d\nu_{BMS}}$ satisfy 
 $$1/C \leq \frac{dg^*\nu_{BMS}}{d\nu_{BMS}} \leq C$$ for all $g \in G$. Note however
that $\frac{d\nu_{BMS}}{d\nu \otimes \nu}$ blows up at the diagonal $\pG \subset\pG \times \pG$. Further, $\partial^2 G$  equipped with $\nu_{BMS}$ is $\sigma-$finite but not finite.
	
\begin{theorem}\cite{bader-furman}\label{thm-doubleerg}
The action of $G$ on $\partial^2 G$ equipped with $\nu_{BMS}$ is ergodic.
\end{theorem}

\begin{prop}\label{pro-existlawimplieslaw}
Suppose that  an {asymptotic law} $L(z',z)$ exists for $\nu_{BMS}-$a.e.\ $(z',z) \in \partial^2 G$. Then $L(z',z)$ is constant a.e.
\end{prop}

\begin{proof}
By Theorem~\ref{thm-doubleerg}, the  $L(z',z)$ must be constant a.e.\ on $\partial^2 G$ equipped with $\nu_{BMS}$.
\end{proof}

\begin{rmk}\label{rmk-existence}
The one line proof of Proposition~\ref{pro-existlawimplieslaw} illustrates the principle that the real issue here is to prove the \emph{existence} of a measurable observable on $\partial^2 G$ equipped with $\nu_{BMS}$. Once such an existence theorem is established, Theorem~\ref{thm-doubleerg} immediately implies that such an observable must be constant a.e.\ on $\partial^2 G$ equipped with $\nu_{BMS}$.
\end{rmk}

We now proceed  to prove the \emph{topological fact}
that the law $L(x,y,z)$ is a continuous function of $(x,y,z) \in \partial^3G /G$:

\begin{theorem}\label{thm-contlaw}
$L(x,y,z)$ is a continuous function of $(x,y,z) \in \partial^3G /G$.
\end{theorem}

\begin{proof} It suffices to show sequential continuity for $\{(x_n,y_n,z_n)\}\to (x,y,z)$. Further, for any such sequence, we may 
	set $\Omega' = \Omega(x,y,z) \cap \bigcap_n  \Omega (x_n,y_n,z_n)$ so that all the (countably many) points $\{x_n,y_n,z_n\}, x,y,z$ are
	non-exceptional for $\omega \in \Omega'$.
 First note that  the full measure set $\Omega'$  may, without loss of generality, be assumed to be $G-$invariant (by taking an intersection with its countably many $G-$translates). Further, by $G-$invariance, 
$L(x,y,z)=L(g.x,g.y,g.z)$. Hence, it suffices to show that 
$L(x,y,z)$ is a continuous function of $(x,y,z) \in \partial^3G$.
It further suffices to prove this continuity coordinate-wise. We thus keep $y, z$ fixed and consider a sequence $x_n \to x$. We want to show that $L(x_n,y,z) \to L(x,y,z)$. By group-invariance we can further assume that the (geometric/deterministic) centroid $c_{xyz}=1$.

By a standard hyperbolic geometry argument, the (geometric/deterministic) centroids $c_{x_nxy}, c_{x_nxz}, c_{x_nx1}$ all lie within a uniformly bounded distance (in terms of $\delta$) of each other. We thus set
$C_n:= c_{x_nx1}$ and observe that $C_n \to x$ as $n \to \infty$. Let $X_n$ denote the random variable equal to the distance of the random
centroid of $1, x_n, x$ from $C_n$.
By Theorem~\ref{thm-centroidexpdecay}, $\P[X_n \geq N] \leq e^{-cN}$ for some uniform $c>0$. In particular, $[x_n,z]_\omega$ and  $[x,z]_\omega$
coalesce within a distance $N$  of $C_n$ with probability at least 
$1-e^{-cN}$. Hence, there exists a set $\Omega_N$ of measure at least $1-e^{-cN}$, such that the backward trees from $c(x_n,y,z,\omega)$ and $c(x,y,z,\omega)$
coincide on the nose for $\omega \in \Omega_N$. Thus, $L(x_n,y,z) \to L(x,y,z)$. 
A similar argument shows that  $L(x,y,z)$ is continuous in the second and third variables too.
\end{proof}

\section{Related Works and Future Directions}
We finish this article with a brief discussion of results related to the ones described in the previous sections and some potential questions of future research.

\subsection{Extra symmetries}\label{sec-ccmt} 
Given Theorem~\ref{thm-contlaw}, it is tempting to upgrade the conclusion to an actual law by demanding an additional group of symmetries. However, in the light of the following theorem from \cite{ccmt}, such a naive approach will not work. We refer the reader to \cite{ccmt} for background notions, particularly for the notion of a hyperbolic locally compact group $L$ and an amenable hyperbolic group. Briefly, a locally compact group $L$ is hyperbolic if $L$ acts by isometries on a hyperbolic space with compact stabilizers.

\begin{theorem}\label{thm-ccmt}
	Let $L$ be a non-amenable hyperbolic locally compact group. If $L$ contains
	a cocompact amenable closed subgroup, then there exists a unique maximal compact normal
	subgroup $W \subset L$, such that 
	\begin{enumerate}
		\item Either $L/W \subset Isom (X)$ where $X$ is a rank one symmetric space of noncompact type, and $L/W$ acts transitively on $X$.
		\item Or $L/W \subset Isom (\TT)$, where $\TT$ is a locally finite
		non-elementary tree. Further, 
		\begin{itemize}
			\item $L/W$ acts faithfully and properly discontinuously by simplicial automorphisms without inversions on $\TT$ with exactly two orbits of vertices,
			\item The $L$ action on the boundary $\partial \TT$ is 2-transitive, i.e.\ it is transitive on $\partial^2 T$. 
		\end{itemize}
	\end{enumerate}
\end{theorem}

\begin{rmk}\label{rmk-tdlc}
	There is an easy way to transition between discrete groups and totally disconnected {locally} compact (tdlc) groups $L$ as follows. Let $L$ be tdlc, $G$ a finitely generated discrete group (e.g.\ $G$ could be free) and let
	$G<L$ be a dense embedding. Let $K<L$ be a compact open and $H=G \cap K$. Then
	the  Schlichting completion $G\/ \! \/ H =L$ and $G$ commensurates $H$. Conversely if $H<G$ is a commensurated subgroup then  $G\/ \! \/ H =L$ is a tdlc group and the closure $\bbar H$ in $L$ is a compact open. We refer the reader to \cite{shalom-willis,bfmv} for details on the  Schlichting completion and commensurated subgroups.  
\end{rmk}

Now, suppose that the given hyperbolic group $\Gamma$ acts properly cocompactly on a (necessarily hyperbolic) graph $\GG$ so that the isometry group $Isom(\GG)$ is indiscrete. Then $L=Isom(\GG)$  is necessarily a tdlc group 
and $G < L$ is a cocompact lattice.
Let $v $ be a vertex of $\GG$ and $K\subset L$ be the stabilizer of $v$. Then the point stabilizer $K$ is  a compact tdlc group providing an extra family of isometries of $\GG$. Certainly, $K$ acts on $\partial G$, and hence on $\partial^3 G$. If $L$ acts with dense orbits on 
$\partial^3 G$, then continuity of the law established in Theorem~\ref{thm-contlaw} implies that we have a genuine (constant) law.

We observe now that the hypothesis that $L$ acts with dense orbits on 
$\partial^3 G$ is highly restrictive. Note first that $K-$orbits on $\pG$ are necessarily compact. Let $U,V, W \subset \pG$ be open subsets such that $U \times V \times W$ is contained in a fundamental domain for the action of $G$ on $\partial^3 G$. By shrinking $U,V, W$ slightly, we may further demand that
$o \in \GG$ is the centroid of any $(u,v,w) \in U \times V \times W$. Since vertices of $\GG$ are in one-to-one correspondence with the cosets of $K$ in $L$ and we have assumed that  $L$ acts with dense orbits on 
$\partial^3 G$, it follows that for any  $(u,v,w) \in U \times V \times W$,
the $K-$orbit $K.(u,v,w)$ contains a subset that is dense in $U \times V \times W$. Next, since $K$ is compact, $K-$orbits are compact, hence closed.
Therefore $K.(u,v,w)\cap  U \times V \times W$ is both dense and closed in
$U \times V \times W$. This forces $ U \times V \times W \subset K.(u,v,w)$.
In particular, $U \subset K.u \subset \pG$. Next, for any $u' \in K.u$, there exists a neighborhood $U$ of $u'$ and two other disjoint subsets $V, W \subset \pG$ (both disjoint from $U$) such that $U \times V \times W$ is contained in a fundamental domain for the action of $G$ on $\partial^3 G$. This forces $K.u$ to be open in
$ \pG$. Thus, $K.u$ is both open and closed in $\pG$. 

Next, assume further that $G$ is one-ended so that $\pG$ is connected. It follows that  $K.u = \pG$. Therefore, $L.u =\pG$, i.e. $L$ acts transitively on $\pG$. The stabilizer $L_\xi$ of $\xi \in \pG$ is a closed subgroup; hence an amenable hyperbolic group in the sense of \cite{ccmt}. Further, $L/L_\xi = \pG$ is compact. Since $L$ is a tdlc, it follows from  Theorem~\ref{thm-ccmt}
that $L$ is naturally a closed subgroup of an isometry group of a tree. In particular, $G$ cannot be one-ended, a contradiction. This leads us to the following questions:

\begin{qn}\label{qn-law} $ $
	\begin{enumerate}
	\item Discrete case: If $G$ is one-ended and the probability distribution $\rho$ on edges is continuous, can the {subtree of the backward random tree
		 $T(\xi,\omega)$ starting at $x\in G$  have a law independent of $x$ ?}
	\item Continuous case: To address Case 1 of Theorem~\ref{thm-ccmt}, develop a theory of radially symmetric $L-$invariant continuous FPP on \emph{rank one symmetric spaces}. It should follow from the above discussion that for a robust enough theory, the 
	backward random tree has a law.
		\end{enumerate}
\end{qn}

We expect that the answer to {Question~\ref{qn-law}(1) is `No' (unlike 
what we expect for Question~\ref{qn-law}(2))}, as the following toy example suggests. Let $G=\Z$ denote the integers, and $S=\{\pm 1, \pm 2\}$ be a generating set. Then the Cayley graph $\Gamma$ of $(G,S)$ consists
of $\Z$ with edges between $(i, i+1)$ and $(i, i+2)$ for all $i$. Let {$\xi=\infty$}. Let $\rho$ denote the probability distribution given  by a Dirac mass at 1. Then the backward tree $T(\xi,\omega)$ from any positive even integer $2n$ consists of two geodesic rays in $\Gamma$ with vertex sets given by:
\begin{enumerate}
	\item $\{2n,2n+2, 2n+4,2n+6,\cdots\}$,
	\item $\{2n,2n+1, 2n+3,2n+5,\cdots\}$.
\end{enumerate}

On the other hand, the backward tree $T(\xi,\omega)$ from any positive odd integer $2n+1$ consists of a unique geodesic ray in $\Gamma$ with vertex set given by $\{2n+1, 2n+3,2n+5,\cdots\}$. {Since all positive integers are contained in the backward tree from 0, there cannot exist a law
	independent of the root.}

\bigskip
\subsection{Future Directions}
~~

\bigskip

\noindent {\bf General geometries:} One natural and interesting direction is to investigate how much of the theory developed for the FPP on hyperbolic groups can be extended to more general settings. As was remarked in \cite{BM24}, many results about geodesics, bi-geodesics and their coalescence can be easily extended to general hyperbolic graphs under mild conditions, whereas the law of large numbers result in Theorem \ref{thm-vexists} crucially depends on the group structure. It would be interesting to see, if these results can also be proved in a continuum setting, such as FPP models based on a Poisson process defined on the hyperbolic plane. It would also be interesting to explore FPP models on negatively curved spaces without using the full power of hyperbolicity. It was proved in \cite{benjamini-tessera} that bigeodesics exist in FPP on hyperbolic graphs with a Morse bi-infinite quasi-geodesic. This result has recently been extended to the case of \emph{sublinearly Morse} quasi-geodesics in \cite{JQ25}. \\

\noindent {\bf CLT and fluctuations:}
Recall that in Theorem \ref{thm-vexists} we showed that the passage times grow linearly with a fixed velocity in almost every direction, whereas in Theorem \ref{t:linvar} we showed that the variance of the passage time between two points grows linearly in the distance between them. Given this, it is natural to conjecture (and in fact, it was conjectured in \cite{BM19}) that the passage times, centered and scaled by the mean and standard deviation, respectively, converge weakly to a Gaussian distribution. A central limit theorem answering this question has recently been proved by Chawla and Gorski \cite{CG25}. It is of course a very interesting question to try and understand the fluctuation and scaling limit for FPP on more general groups and graphs, and decipher connections between the geometry of the ambient space and the fluctuations of the FPP metric. Recall that for FPP on $\Z^2$, the variance between two points at distance $n$ is expected to grow like $n^{2/3}$ (see \cite{ADH}), and to prove this remains one of the most important open questions in the area. \\

\noindent {\bf Exceptional directions and N3G:}
Finally, we discuss questions regarding exceptional directions for geodesic trees. Recall from Theorem \ref{thm-omni-BM24} that if $G$ is not virtually free then exceptional directions exist both for the forward geodesic trees (the tree $F(v,\omega)$ consisting of the semi-infinite FPP geodesics starting from a fixed point $v$), as well as the backward geodesic trees (the trees $T(\xi,\omega)$ consisting of the coalescing geodesic rays in a fixed boundary direction $\xi$). Consider the set of exceptional directions in the forward geodesic tree $F(v,\omega)$ for a given realization $\omega$ of the FPP metric. 

\begin{qn}
Are these exceptional directions independent of the base-point $v$?
\end{qn} 

This question is related to the behavior of exceptional directions with respect to the $G-$action. A result showing that exceptional directions are independent of the base point was proved in \cite{jrs} in the context of the exactly solvable model of planar exponential last passage percolation. Another interesting direction of research is to understand the multiplicities of exceptional directions. Recall from Theorem \ref{thm-omni-BM24}, (4) and (6), that the maximum number of disjoint geodesics in an exceptional direction is at least $\dim_t \pG + 1$  whereas we only have a non-explicit absolutely constant upper bound depending on the $G$ and the passage time distribution. It is interesting to understand whether the lower bound is tight, i.e.\

\begin{qn}
Is it true that almost surely there are no exceptional directions with 
multiplicity greater than $\dim_t \pG + 1$?
\end{qn}

Following a general conjecture from \cite{coupier2023thicktraceinfinityhyperbolic} (see also \cite{coupier11,jrs} for a proof in the context of planar exponential LPP) it is expected that the lower bound is indeed tight for planar hyperbolic FPP (the case covered in Theorem \ref{thm-omni-BM24}, (3)). Thus one expects that almost surely there are no exceptional directions with multiplicity $3$ or more. In an ongoing joint work with Chawla, we are studying the no three geodesics (N3G) phenomenon in planar hyperbolic FPP, but the question remains out of reach in higher dimensions.

{
\section*{Acknowledgments}
The authors would like to thank the anonymous referee for carefully and diligently going through an earlier version of this article. The detailed comments in his/her report have been extremely helpful to the improvement of this paper.
}

\bibliography{fpphoro}
\bibliographystyle{alpha}

\end{document}